\documentclass[a4paper,12pt]{article}

\usepackage{dsfont, amsmath, amsthm, amsfonts, amssymb,accents}
\usepackage{graphicx,color}

\theoremstyle{plain}
\newtheorem{thrm}{Theorem}[section]

\newtheorem{lmm}{Lemma}[section]

\theoremstyle{remark}
\newtheorem{rmrk}{Remark}[section]

\numberwithin{equation}{section}

\def\tht{\theta}
\def\Om{\Omega}
\def\om{\omega}
\def\e{\varepsilon}
\def\g{\gamma}
\def\G{\Gamma}
\def\l{\lambda}
\def\p{\partial}
\def\D{\Delta}

\def\vk{\varkappa}

\def\a{\alpha}
\def\b{\beta}

\def\d{\delta}

\def\z{\zeta}

\def\Odr{\mathcal{O}}
\def\H{W_2}

\def\Wo{W_{2,0}}
\def\k{\kappa}

\def\di{\,\mathrm{d}}

\def\vc{\boldsymbol}

\allowdisplaybreaks

\begin{document}
\title{Complete asymptotic expansions for eigenvalues of
Dirichlet Laplacian in thin three-dimensional rods}

\def\thefootnote{}

\author{D. Borisov$^1$, \footnote{D.B. was partially supported by RFBR (10-01-00118), 
by the grant of the President of Russia for young scientists-doctors
of sciences (MD-453.2010.1) and for Leading Scientific School
(NSh-6249.2010.1), and by the grant of FCT (ptdc/mat/101007/2008).}
  and G. Cardone$^2$}
\date{}
\maketitle

\begin{quote}
{\small {\em 1) Bashkir State Pedagogical University, October
Revolution
\\
\phantom{1) } St.~3a, 450000 Ufa, Russia
\\
2) University of Sannio, Department of Engineering,
\\
\phantom{2) } Corso Garibaldi, 107, 82100 Benevento, Italy }
\\
\phantom{1) }\texttt{borisovdi@yandex.ru},
\texttt{giuseppe.cardone@unisannio.it}}
\end{quote}

\begin{abstract}
We consider Dirichlet Laplacian in a thin curved three-dimensional
rod. The rod is finite. Its cross-section is constant and small, and
rotates along the reference curve in an arbitrary way. We find a
two-parametric set of the eigenvalues of such operator and construct
their complete asymptotic expansions. We show that this
two-parametric set contains any prescribed number of the first
eigenvalues of the considered operator. We obtain the complete
asymptotic expansions for the eigenfunctions associated with these
first eigenvalues.
\end{abstract}
%



 \maketitle
\section*{Introduction}

The asymptotics of the spectra of elliptic operators in thin domains
were studied by many authors, see, for instance,
\cite{BF1}-\cite{Kr}, 
\cite{Na4}, \cite{Pa3}, and the references therein. There are two
types of thin domains usually considered, namely, thin rods and thin
plates. Both types were considered in the book \cite{Na4}. The
eigenvalues for elliptic operators with Neumann boundary condition
on the lateral surface of the rods and on the bases of the plates
were studied. The asymptotic expansions for the eigenvalues and the
eigenfunctions were constructed and justified. We also mention the
survey \cite{Gr} on thin rods.

A thin two-dimensional domain formed by two different thin
rectangles was studied in \cite{Pa3}, i.e., the boundary of the
domain was non-smooth. The operator considered was Laplacian subject
to Neumann condition on the bases and to Dirichlet one on the
lateral boundary. The main results were the asymptotic expansions
for the eigenvalues remaining bounded as the width of the domain
tends to zero, and the asymptotics for the associated
eigenfunctions. These asymptotic expansions were rigorously
justified. Laplacian in a thin two-dimensional domain was also
considered in \cite{FS1}. The domain had a variable width with the
unique point of maximum. The uniform resolvent convergence was
established and two-terms asymptotics for the eigenvalues were
obtained, as well as convergence theorems for the associated
eigenfunctions. In \cite{FS2} these results were extended for an
infinite thin strip under similar conditions for the width. We also
mention the paper \cite{Kr}, where a thin strip (finite or infinite)
was considered with Neumann condition on the upper boundary and with
Dirichlet condition on the lower boundary. Here two-terms
asymptotics for the first eigenvalues were constructed. The case of
a curved infinite strip was also studied in \cite{DE}, where the
number of the discrete eigenvalues below the essential spectrum was
estimated. The results of \cite{FS1} were also extended in
\cite{BF1}. Here a two-parametric set of the eigenvalues was found
and their complete asymptotic expansions were constructed.

A finite three-dimensional rod was considered in \cite{BMT}. The
cross section was supposed to be constant and to rotate along the
reference curve in an arbitrary way. Two-terms asymptotics for the
first eigenvalues were constructed and convergence theorems for the
associated eigenfunctions were established. Similar results  were
obtained in \cite{FK} for a tube in a space of arbitrary dimension.
An infinite three-dimensional thin tube with a round cross section
was studied in \cite{DE}. The number of the discrete eigenvalues
below the essential spectrum was estimated and their complete
asymptotic expansions were constructed. We also mention the paper
\cite{CDN}, where a multi-dimensional thin cylinder with distorted
ends was considered. The operator studied was Laplacian in such
domain subject to Dirichlet condition on the lateral surface and to
Neumann one on the distorted ends. The attention was paid to the
localization effect of some eigenfunctions at the distorted ends.
The asymptotics for these eigenfunctions and the corresponding
eigenvalues were constructed.

In this paper we extend the results of \cite{BMT}. We again consider
Dirichlet Laplacian in a curved thin rod. The cross section of the
rod is a fixed domain, which rotates along the reference curve in an
arbitrary way. In what follows this operator, its eigenvalues and
eigenfunctions are referred to as the perturbed ones. We find a
two-parametric set of the perturbed eigenvalues and construct their
complete asymptotic expansions. The eigenvalues are indexed by the
first two terms of their asymptotic expansions. Namely, the leading
terms are determined by the eigenvalues of Dirichlet Laplacian on
the cross-section of the rod. Each of the leading terms determines a
certain operator on the reference curve, and its eigenvalues are the
next-to-leading terms of the aforementioned asymptotic expansions
for the perturbed eigenvalues. It is convenient to group the
perturbed eigenvalues into a countable set of series, where each
series consists of the perturbed eigenvalues with the same leading
term in the asymptotic expansions. We show that the series
associated with the smallest leading term contains any prescribed
number of the first eigenvalues of the perturbed operator provided
the rod is thin enough.
%
We prove that these eigenvalues are simple and construct complete
asymptotic expansions for the associated eigenfunctions.

In conclusion to this section, we describe briefly the contents of
the paper. The next section contains the description of the problem
and the main results. In the third section we introduce a change of
variables required for the constructing the asymptotic expansions.
In the fourth section we select the aforementioned two-parametric
series of the eigenvalues and construct their asymptotic expansions.
In the last fifth section we describe the first eigenvalues of the
perturbed operator and give the asymptotic expansions for the
associated eigenfunctions.

\section{Formulation of the problem and the main results}

Let $x=(x_1,x_2,x_3)$ be Cartesian coordinates in $\mathds{R}^3$,
$\g$ be a finite infinitely differentiable curve in $\mathds{R}^3$
without self-intersections. By $s$ and $s_0$ we denote the arc
length and the length of $\g$, $s\in[0,s_0]$. We parameterize $\g$
by its arc length, and $\vc{r}=\vc{r}(s)$ is the infinitely
differentiable vector describing $\g$. The tangential vector of $\g$
is indicated by $\vc{\tau}=\vc{\tau}(s)$. By
$\vc{\eta}=\vc{\eta}(s)$ we denote an infinitely differentiable in
$s\in[0,s_0]$ unit vector defined on $\g$ being orthogonal to
$\vc{\tau}(s)$ for all $s\in[0,s_0]$. We let
$\vc{\b}(s):=[\vc{\tau}(s),\vc{\eta}(s)]$. It is clear that
$\vc{\b}(s)$ is infinitely differentiable in $s\in[0,s_0]$, and
$(\vc{\tau},\vc{\eta},\vc{\b})$ is an orthonormalized frame on $\g$.
One of the possible choices of $\vc{\eta}$ is
\begin{equation}\label{1.0a}
\vc{\eta}(s):=\cos \a(s) \vc{n}(s)+\sin\a(s)\vc{b}(s),
\end{equation}
where $\vc{n}=\vc{n}(s)$ and $\vc{b}=\vc{b}(s)$ are the normal and
binormal vectors of $\g$, and $\a(s)\in C^\infty[0,s_0]$ is an
arbitrary function describing how our frame rotates with respect to
the Frenet one. It follows from (\ref{1.0a}) that
\begin{equation}\label{1.0c}
\vc{\b}(s):=-\sin\a(s) \vc{n}(s)+\cos\a(s)\vc{b}(s).
\end{equation}
Although this formula and (\ref{1.0a}) could be an appropriate
definition of $\vc{\eta}$ and $\vc{\b}$, we do not use this way. The
reason is that the Frenet frame does not exists for all smooth
curves, since the normal vector can be undefined at the points where
$\vc{r}''(s)=0$.

By $\om$ we indicate a bounded domain in $\mathds{R}^2$ with
infinitely smooth boundary, and the symbol $\e$ stands for a small
positive parameter. We introduce a thin curved rod as
\begin{equation*}
\Om_\e:=\{x\in \mathds{R}^3: x=\vc{r}(s)+\e \xi_2
\vc{\eta}(s)+\e\xi_3 \vc{\b}(s), s\in(0,s_0), (\xi_2,\xi_3)\in\om\}.
\end{equation*}
Since the curve $\g$ is smooth and not self-intersecting, the rod
$\Om_\e$ has no self-intersections for $\e$ small enough.
Hereinafter the parameter $\e$ is assumed to be chosen exactly in
such way.

The main object of our study is the spectrum of the Dirichlet
Laplacian in $L_2(\Om_\e)$, and this operator is denoted by
$\mathcal{H}_\e$. We introduce this operator rigourously as the
Friedrichs extension of $-\D_x$ on $C_0^\infty(\Om_\e)$. For each
$\e>0$ the operator $\mathcal{H}_\e$ has a compact resolvent and its
spectrum is thus purely discrete. The aim of this paper is to
construct the complete asymptotic expansions for the eigenvalues of
$\mathcal{H}_\e$. We also observe that the eigenvalues of
$\mathcal{H}_\e$ can be equivalently regarded as those of the
boundary value problem
\begin{equation*}
-\D\psi_\e=\l_\e\psi_\e\quad\text{in}\quad\Om_\e,\qquad
\psi_\e=0\quad\text{on}\quad\p\Om_\e,\qquad \psi_\e\in\H^1(\Om_\e).
\end{equation*}

In order to formulate the main results we need to introduce
additional notations. Let $\Wo^2(\om)$ be the subspace of
$\H^2(\om)$ consisting of the functions vanishing on $\p\om$. In the
same way we introduce the space $\Wo^2(0,s_0)$. By $\mathcal{S}$ we
indicate Dirichlet Laplacian in $L_2(\om)$ with $\Wo^2(\om)$ as the
domain. This operator is self-adjoint. Let $\l_n$ be the eigenvalues
of $\mathcal{S}$ arranged in the ascending order with the
multiplicities taken into account,
\begin{equation*}
\l_1<\l_2\leqslant \l_3\leqslant\ldots\leqslant \l_n\leqslant \ldots
\end{equation*}
By $\phi_n$ we denote the associated eigenfunctions orthonormalized
in $L_2(\om)$. By the smoothness improving theorems \cite[Ch. I\!V,
Sec. 2.3]{M} the functions $\phi_n$ are infinitely differentiable in
$\overline{\om}$.

Let an eigenvalue $\l_n$ be simple. Denote
\begin{equation}\label{1.0b}
\mathcal{R}:=\xi_3\frac{\p}{\p \xi_2}-\xi_2\frac{\p}{\p\xi_3},\quad
C_n(\om):=\int\limits_{\om} |\mathcal{R}\phi_n|^2\di\xi.
\end{equation}

It is straightforward to check that
\begin{equation}\label{2.1a}
\vc{\tau}'=\k_1\vc{\eta}-\k_2\vc{\b}, \quad
\vc{\eta}'=-\k_1\vc{\tau}+\k_3\vc{\b}, \quad
\vc{\b}'=\k_2\vc{\tau}-\k_3\vc{\eta},
\end{equation}
where $\k_i=\k_i(s)\in C^\infty[0,s_0]$ are certain functions
characterizing the geometric properties of $\g$ and rotation of
$\vc{\eta}$. By $\mathcal{L}_n$ we denote the operator
\begin{equation*}
-\frac{d^2}{ds^2}+C_n(\om)\k_3^2(s)-\frac{\k_1^2(s)+\k_2^2(s)}{4}
\end{equation*}
in $L_2(0,s_0)$ with the domain $\Wo^2(0,s_0)$. The operator
$\mathcal{L}_n$ is self-adjoint. Since the operator $\mathcal{L}_n$
is one-dimensional, by Cauchy theorem we conclude that all its
eigenvalues are simple. We indicate these eigenvalues by
$\l_0^{(n,m)}$, $m=1,2,\ldots$ and arrange them in the ascending
order with the multiplicities taken into account. Let
$\Psi_0^{(n,m)}$ be the associated eigenfunctions orthonormalized in
$L_2(0,s_0)$.

If one chooses $\vc{\eta}$ and $\vc{\b}$ in accordance with
(\ref{1.0a}), (\ref{1.0c}), it yields
\begin{equation*}
\k_1=\k\cos\a,\quad \k_2=\k\sin\a,\quad \k_3=\a'+\vk,
\end{equation*}
where $\k=\k(s)$ and $\vk=\vk(s)$ are the curvature and torsion of
$\g$, and $\k,\vk\in C^\infty[0,s_0]$. In this case the operator
$\mathcal{L}_n$ becomes
\begin{equation*}
-\frac{d^2}{ds^2}+C_n(\om)(\a'(s)+\vk(s))^2-\frac{\k^2(s)}{4}.
\end{equation*}

Our first result gives the complete asymptotic expansions for the
eigenvalues of $\mathcal{H}_\e$.

\begin{thrm}\label{th1.1}
Let $\l_n$ be a simple eigenvalue of $\mathcal{S}$.  Then there
exists a two-parametric set of the eigenvalues $\l^{(n,m)}_\e$ of
$\mathcal{H}_\e$ with the asymptotics
\begin{equation}
\l^{(n,m)}_\e=\e^{-2}\l_n+\l_0^{(n,m)}+\sum\limits_{i=1}^{\infty}
\e^i\l_i^{(n,m)},\label{1.1}
\end{equation}
where 
\begin{equation}\label{1.2}
\begin{aligned}
&\l_1^{(n,m)}=\big( \psi_0^{(n,m)},Q^{(n,m)}\psi_0^{(n,m)}
\big)_{L_2(\Om)}+2\big(\mathcal{R}\psi_0^{(n,m)},\k_3^2\mathrm{q}
\mathcal{R}\psi_0^{(n,m)}\big)_{L_2(\Om)},
\\
&Q^{(n,m)}:=
\left(2\l_0^{(n,m)}-\left(2C_n(\om)-\frac{1}{2}\right)\k_3^2\right)
\mathrm{q}+\frac{1}{2}\frac{\p^2 \mathrm{q}}{\p s^2} +\frac{1}{2}
\k'_3\mathcal{R}\mathrm{q},
\\
&\mathrm{q}= \mathrm{q}(s,\xi):=\k_1(s)\xi_2-\k_2(s)\xi_3,\quad
\psi_0^{(n,m)}= \psi_0^{(n,m)}(s,\xi):=\Psi_0^{(n,m)}(s)\phi_n(\xi).
\end{aligned}
\end{equation}
The remaining coefficients in the asymptotic series (\ref{1.1}) are
given by the formulas (\ref{3.42}) in Lemma~\ref{lm3.1}.
\end{thrm}

The series (\ref{1.1}) are the asymptotic ones and we know nothing
on their convergence.

We observe that the set of the eigenvalues described in
Theorem~\ref{th1.1} is two-parametric and is indexed by $n$ and $m$.
Given $n$, the eigenvalues $\l^{(n,m)}(\e)$ form a series with the
same leading term. We stress that Theorem~\ref{th1.1} does not imply
that the eigenvalues $\l_\e^{(n,m)}$ form the whole set of the
eigenvalues of $\mathcal{H}_\e$. The first reason is the assumption
on the simplicity of $\l_n$. And even without this assumption it is
an independent problem to find out whether the eigenvalues
$\l_\e^{(n,m)}$ are the only possible ones or not.

We make the assumption that $\l_n$ is simple in order to simplify
the calculations in the formal constructing of the asymptotic
expansion, see Sec. 3. If $\l_n$ is multiply, it is also possible to
construct the asymptotic expansions, but the formal constructing
becomes more complicated and requires some additional careful
calculations. Another interesting issue is the multiplicities of the
perturbed eigenvalues corresponding to a multiple eigenvalue $\l_n$.
In view of these issues we regard the case of multiple eigenvalue
$\l_n$ as an independent problem, which we postpone for another
article.

One more interesting question is on the asymptotic expansions for
the eigenfunctions associated with $\l_\e^{(n,m)}$. As usually, to
justify such asymptotic expansions, one has to know lower bounds for
the distances between the perturbed eigenvalues. The structure of
the eigenvalues $\l_\e^{(n,m)}$ is such that it is rather difficult
to obtain such bounds once the eigenvalues are bigger than
$\e^{-2}\l_2$. If we consider only the first eigenvalues of
$\mathcal{H}_\e$ lying between $\e^{-2}\l_1$ and $\e^{-2}\l_2$, it
is possible to prove the mentioned lower bounds and to obtain the
asymptotic expansions for the associated eigenfunctions. This is our
second main result. Before formulating it, we introduce two
additional notations,
\begin{equation*}
\Om:=\{(s,\xi): 0<s<s_0, \xi\in\om\},\quad \Om^{(t)}:=\{(s,\xi):
t<s<s_0-t, \xi\in\om\},\quad t\in(0,s_0/2).
\end{equation*}


\begin{thrm}\label{th1.2}
Given any $M\geqslant 1$, there exists $\e_0=\e_0(M)>0$ such that
for all $\e<\e_0$ the first $M$ eigenvalues of $\mathcal{H}_\e$ are
$\l_\e^{(1,m)}$, $m=1,\ldots,M$,  satisfying (\ref{1.1}). These
eigenvalues are simple and the associated eigenfunctions have the
asymptotics
\begin{equation}\label{1.3}
\psi_\e^{(1,m)}(x(s,\e\xi))=\Psi_0^{(1,m)}(s)\phi_1(\xi)
+\sum\limits_{i=1}^{\infty} \e^i\psi_i^{(1,m)}(s,\xi),
\end{equation}
where the coefficients of the series are given in Lemma~\ref{lm3.1}.
The asymptotics hold true in $\H^1(\Om)$-norm and
$C^k(\overline{\Om^{(t)}})$-norms for all $k\geqslant 0$,
$t\in(0,s_0/2)$.
\end{thrm}

The results of \cite{BMT} consist of the two-term asymptotics for
$\l_\e^{(1,m)}$ and the leading term in the asymptotics for the
eigenfunctions $\psi_\e^{(1,m)}$. The asymptotics for the
eigenfunctions were obtained in $L_2(\Om)$. Theorem~\ref{th1.2}
extends these results in two directions. First, it gives the
complete asymptotic expansions. Second, the asymptotics of the
eigenfunctions are given in a stronger norm. One more extension is
provided by Theorem~\ref{th1.1}. Namely, in addition to the first
series of the eigenvalues $\l_\e^{(1,m)}$ described in \cite{BMT},
we describe a countable set of similar series $\l_\e^{(n,m)}$,
$n\geqslant 2$.

One more difference from \cite{BMT} is the technique employed. The
study in \cite{BMT} was based on $\G$-convergence of certain
functionals. Our approach consists of two main steps. The first step
is the formal constructing of the asymptotic expansions for the
eigenvalues and the eigenfunctions by the multiscale method
\cite{BP}. The second step is an estimating of error terms by a
result from spectral perturbation theory, see \cite[Lms 12,13]{VL}.

\section{Change of variables}

In this section we transform the operator $\mathcal{H}_\e$ to
another one which will be more convenient in proving
Theorems~\ref{th1.1},~\ref{th1.2}.

Let $y=(y_2,y_3)$ be Cartesian coordinates in the plane spanned over
$\vc{\eta}$ and $\vc{\b}$ with the axes along these vectors and so
that the variable $y_2$ corresponds to $\vc{\eta}$ and the variable
$y_3$ does to $\vc{\b}$. We first pass to the variables $(s,y)$ and
the domain $\Om_\e$ is mapped onto
\begin{equation*}
\widetilde{\Om}_\e:=\{(s,y): 0<s<s_0, \e^{-1}y\in\om\}.
\end{equation*}
At the second step we rescale the variables $y$ passing to variables
$(s,\xi)$, where $\xi=(\xi_2,\xi_3)=\e^{-1} y$. Then the domain
$\widetilde{\Om}_\e$ is mapped onto $\Om$.

We introduce the operator describing the passing to the variables
$(s,y)$ as
\begin{equation*}
(\mathcal{U}u)(s,y)= u(x(s,y)),\quad (\mathcal{U}^{-1}u)(x)=
u(s(x),y(x)).
\end{equation*}
We let
\begin{equation*}
\widetilde{\mathcal{H}}_\e:=\mathrm{p}\,\mathcal{U} \mathcal{H}_\e
\mathcal{U}^{-1},\quad
\mathrm{p}=\mathrm{p}(s,y):=1-\mathrm{q}(s,y),\quad
\mathrm{q}(s,y)=\k_1(s)y_2-\k_2(s)y_3.
\end{equation*}
If $\l_\e$ and $\psi_\e$ are an eigenvalue and an associated
eigenfunction of $\mathcal{H}_\e$, it is clear that the function
$\mathcal{U}\psi_\e$ solves the equations
\begin{equation}\label{2.0b}
\mathcal{U}\mathcal{H}_\e \mathcal{U}^{-1}\,
\mathcal{U}\psi_\e=\l_\e\, \mathcal{U}\psi_\e,\quad
\widetilde{\mathcal{H}}_\e \, \mathcal{U}\psi_\e=\l_\e\mathrm{p}\,
\mathcal{U}\psi_\e.
\end{equation}

Let us obtain the differential expression for
$\widetilde{\mathcal{H}}_\e$. Taking into account (\ref{2.1a}), and
differentiating the identity
\begin{equation*}
x=\vc{r}(s)+y_2\vc{\eta}(s)+y_3\vc{\b}(s)
\end{equation*}
with respect to $s$, $y_2$, and $y_3$, we obtain
\begin{equation*}
\frac{\p x}{\p s}=\mathrm{p}\vc{\tau}(s)- \k_3 y_3\vc{\eta}(s)+
\k_3y_2\vc{\b}(s),\quad \frac{\p x}{\p y_2}=\vc{\eta}(s),\quad
\frac{\p x}{\p y_3}=\vc{\b}(s).
\end{equation*}
Thus, the derivatives with respect to $x$ and $(s,y)$ are related by
the identity
\begin{equation*}
\nabla_{(s,y)}=\mathrm{P}\nabla_x,
\end{equation*}
where $\mathrm{P}$ is the matrix with the rows
\begin{equation*}
\mathrm{P}:=
\begin{pmatrix}
\mathrm{p}\vc{\tau}- \k_3y_3\vc{\eta}+ \k_3y_2\vc{\b}
\\
\vc{\eta}
\\
\vc{\b}
\end{pmatrix}.
\end{equation*}
It is easy to check that
\begin{equation}
\det\mathrm{P}=\mathrm{p},\quad
\nabla_{x}=\mathrm{P}^{-1}\nabla_{(s,y)}, \quad \mathrm{P}^{-1}=
\begin{pmatrix}
\mathrm{p}^{-1}\vc{\tau}\ \ \k_3 y_3
\mathrm{p}^{-1}\vc{\tau}+\vc{\eta}\ \ -\k_3 y_2
\mathrm{p}^{-1}\vc{\tau}+\vc{\b}
\end{pmatrix},
\label{2.1}
\end{equation}
where the vectors in the definition of $\mathrm{P}^{-1}$ are treated
as columns. Since $\e^{-1}y\in\om$, we have $y=\Odr(\e)$ and hence
$\mathrm{q}(s,y)=\Odr(\e)$. It yields that the function
$\mathrm{p}(s,y)$ is strictly positive for the considered values of
$y$.

By (\ref{2.1}) for each $u_1, u_2\in C_0^\infty(\widetilde{\Om}_\e)$
we have
\begin{align*}
(\widetilde{\mathcal{H}}_\e u_1, u_2)_{L_2(\widetilde{\Om}_\e)}&=
(\mathrm{p}\,\mathcal{U}\mathcal{H}_\e \mathcal{U}^{-1}
u_1,u_2)_{L_2(\widetilde{\Om}_\e)} = (\mathcal{H}_\e
\mathcal{U}^{-1} u_1, \mathcal{U}^{-1} u_2)_{L_2(\Om_\e)}
\\
& =(\nabla_x
\mathcal{U}^{-1}u_1,\nabla_x\mathcal{U}^{-1}u_2)_{L_2(\Om_\e)} =
(\mathrm{P}^{-1}\nabla_{(s,y)}u_1, \mathrm{p}\,
\mathrm{P}^{-1}\nabla_{(s,y)}u_2)_{L_2(\widetilde{\Om}_\e)}
\\
&
=-(\mathrm{div}_{(s,y)} \mathrm{p}\,(\mathrm{P}^{-1})^t
\mathrm{P}^{-1}\nabla_{(s,y)}u_1, u_2)_{L_2(\widetilde{\Om}_\e)},
\end{align*}
and therefore
\begin{gather}
\widetilde{\mathcal{H}}_\e:=-\mathrm{div}_{(s,y)}\,
\mathrm{A}\nabla_{(s,y)}\label{2.2a}
\\
\mathrm{A}=(A_{ij})_{i,j=\overline{1,3}}=
\begin{pmatrix}
\mathrm{p}^{-1} & \k_3 y_3\mathrm{p}^{-1} & -\k_3 y_2
\mathrm{p}^{-1}
\\
\k_3 y_3\mathrm{p}^{-1} & \mathrm{p}+\k_3^2 y_3^2\mathrm{p}^{-1} &
-\k_3^2 y_2y_3 \mathrm{p}^{-1}
\\
-\k_3 y_2 \mathrm{p}^{-1} & -\k_3^2 y_2y_3\mathrm{p}^{-1} &
\mathrm{p}+\k_3^2 y_2^2 \mathrm{p}^{-1}
\end{pmatrix}. \nonumber
\end{gather}
Now we pass to the variables $\xi$. It leads us to a finally
transformed operator
\begin{gather}
\widehat{\mathcal{H}}_\e=- \bigg(\frac{\p}{\p s} A_{11}^{(\e)}
\frac{\p}{\p
s}+\e^{-1}\sum\limits_{i=2}^{3}\frac{\p}{\p\xi_i}A_{i1}^{(\e)}
\frac{\p}{\p s} +\e^{-1}\sum\limits_{i=2}^{3}\frac{\p}{\p s}A_{1i
}^{(\e)} \frac{\p}{\p\xi_i}+\e^{-2}\sum\limits_{i,j=2}^{3}
\frac{\p}{\p\xi_i}A_{ij}^{(\e)} \frac{\p}{\p\xi_j}\bigg),
\label{2.2}
\\
\mathrm{A}^{(\e)}:=\big(A_{ij}^{(\e)}\big)_{i,j=\overline{1,3}}=
\begin{pmatrix}
\mathrm{p}_\e^{-1} & \e\k_3\xi_3\mathrm{p}_\e^{-1} & -\e\k_3\xi_2
\mathrm{p}_\e^{-1}
\\
\e\k_3\xi_3\mathrm{p}_\e^{-1} &
\mathrm{p}_\e+\e^2\k_3^2\xi_3^2\mathrm{p}_\e^{-1} &
-\e^2\k_3^2\xi_2\xi_3 \mathrm{p}_\e^{-1}
\\
-\e\k_3\xi_2 \mathrm{p}_\e^{-1} &
-\e^2\k_3^2\xi_2\xi_3\mathrm{p}_\e^{-1} &
\mathrm{p}_\e+\e^2\k_3^2\xi_2^2 \mathrm{p}_\e^{-1}
\end{pmatrix},\label{2.3}
\end{gather}
where $\mathrm{p}_\e(s,\xi):=1-\e \mathrm{q}(s,\xi)$, and the
operator $\widehat{\mathcal{H}}_\e$ is considered in $L_2(\Om)$ as
the Friedrichs extension from $C_0^\infty(\Om)$. The eigenvalue
equation (\ref{2.0b}) casts into the form
\begin{equation}\label{2.4}
\widehat{\mathcal{H}}_\e\psi_\e=\l_\e \mathrm{p}_\e\psi_\e,
\end{equation}
where we redenoted $(\mathcal{U}\psi_\e)(s,\e\xi)$ by $\psi_\e$.

\section{Proof of Theorem~\ref{th1.1}}

In this section we prove Theorem~\ref{th1.1}. The proof is divided
into two parts, the first being devoted to the formal constructing
of the asymptotic expansions. The second part consists in proving
the existence of the mentioned two-parametric set of the eigenvalues
and in the justification of the formal asymptotic expansions for
these eigenvalues, i.e., establishing estimates for the error terms.

We construct the asymptotic expansions for the eigenvalues and the
associated eigenfunctions as the series
\begin{align}
&\l^{(n,m)}_\e=\sum\limits_{i=-2}^{\infty}
\e^i\l_i^{(n,m)},\label{3.1}
\\
&\psi^{(n,m)}_\e(x)=\sum\limits_{i=0}^{\infty}
\e^i\psi_i^{(n,m)}(s,\xi).\label{3.2}
\end{align}
The aim of the formal constructing is to determine the coefficients
of these series.

We expand the functions $A_{ij}^{(\e)}$ in the powers of $\e$,
\begin{align}
&A_{ij}^{(\e)}=\sum\limits_{k=0}^{\infty}\e^k A_k^{(ij)},\quad
A_k^{(ij)}=A_k^{(ji)}, \label{3.3}
\\
&
\begin{aligned}
&A_k^{(11)}=\mathrm{q}^k, && k\geqslant 0,
\\
&A_0^{(12)}=0,\quad A_k^{(12)}=\k_3\xi_3 \mathrm{q}^{k-1}, &&
k\geqslant 1,
\\
&A_0^{(13)}=0,\quad A_k^{(13)}=-\k_3\xi_2 \mathrm{q}^{k-1}, &&
k\geqslant 1,
\\
&A_0^{(22)}=1,\quad A_1^{(22)}=-\mathrm{q},\quad
A_k^{(22)}=\k_3^2\xi_3^2\mathrm{q}^{k-2}, && k\geqslant 2,
\\
&A_0^{(23)}=A_1^{(23)}=0,\quad
A_k^{(23)}=-\k_3^2\xi_2\xi_3\mathrm{q}^{k-2}, && k\geqslant 2,
\\
&A_0^{(33)}=1,\quad A_1^{(33)}=-\mathrm{q},\quad
A_k^{(33)}=\k_3^2\xi_2^2\mathrm{q}^{k-2}, && k\geqslant 2.
\end{aligned}\label{3.4}
\end{align}
Hereinafter $\mathrm{q}=\mathrm{q}(s,\xi)$, if else is not
specified.

We substitute (\ref{2.2}), (\ref{3.1}), (\ref{3.2}), (\ref{3.3}),
(\ref{3.4}) into (\ref{2.4}) and equate the coefficients at the same
powers of $\e$. Calculating the coefficient at $\e^{i-2}$,
$i\geqslant 0$, we obtain
%
\begin{align}
&
\begin{aligned}
(-\D_\xi-\l_{-2}^{(n,m)})\psi_i^{(n,m)}&=
\sum\limits_{j=1}^i\l_{j-2}^{(n,m)}\psi_{i-j}^{(n,m)}
+F_i^{(n,m)}\quad \text{in}\quad \Om,
\\
\psi_i^{(n,m)}&=0\quad \text{on}\quad \p\Om,
\end{aligned}
 \label{3.5}
\\
&F_i^{(n,m)}:=\sum\limits_{j=1}^i
(\mathcal{F}_j-\l_{j-3}^{(n,m)}\mathrm{q})\psi_{i-j}^{(n,m)},
\label{3.5a}
\\
&
\begin{aligned}
\mathcal{F}_1:=&-\frac{\p}{\p\xi_2}\mathrm{q}\frac{\p}{\p\xi_2}
-\frac{\p}{\p\xi_3}\mathrm{q}\frac{\p}{\p\xi_3},
\\
\mathcal{F}_j:=&\frac{\p}{\p s}A_{j-2}^{(11)}\frac{\p}{\p
s}+\sum\limits_{l=2}^{3}\frac{\p}{\p\xi_l}A_{j-1}^{(l1)}\frac{\p}{\p
s} +\sum\limits_{l=2}^{3}\frac{\p}{\p
s}A_{j-1}^{(1l)}\frac{\p}{\p\xi_l}
+\sum\limits_{t,l=2}^{3}\frac{\p}{\p\xi_t}A_{j}^{(tl)}\frac{\p}{\p\xi_l}
\\
=&\frac{\p}{\p s}\mathrm{q}^{j-2}\frac{\p}{\p s}+
\mathcal{R}\k_3\mathrm{q}^{j-2}\frac{\p}{\p s}+ \frac{\p}{\p
s}\k_3\mathrm{q}^{j-2}\mathcal{R}+
\k_3^2\mathcal{R}\mathrm{q}^{j-2}\mathcal{R}, \quad j\geqslant 2.
\end{aligned}\label{3.7}
\end{align}

Consider the problem (\ref{3.5}) for $i=0$. It is clear that its
solution can be chosen as
\begin{equation}\label{3.8}
\psi_0^{(n,m)}(s,\xi)=\Psi_0^{(n,m)}(s)\phi_n(\xi),\quad
\l_{-2}^{(n,m)}=\l_n,
\end{equation}
where the function $\Psi_0^{(n,m)}$ is unknown and should satisfy
the boundary conditions
\begin{equation}\label{3.9}
\Psi_0^{(n,m)}(0)=\Psi_0^{(n,m)}(s_0)=0.
\end{equation}

Taking into account (\ref{3.8}) and an obvious formula
\begin{equation}\label{3.8a}
(\mathcal{F}_1-\l_n\mathrm{q})\phi_n=-\left(
\k_1\frac{\p}{\p\xi_2}-\k_2\frac{\p}{\p\xi_3}\right)\phi_n,
\end{equation}
we write the problem (\ref{3.5}) for $i=1$,
\begin{equation}\label{3.12}
\begin{aligned}
(-\D_\xi-\l_n)\psi_1^{(n,m)}=&-\Psi_0^{(n,m)}\left(
\k_1\frac{\p}{\p\xi_2}-\k_2\frac{\p}{\p\xi_3}\right)\phi_n
 +\l_{-1}^{(n,m)} \Psi_0^{(n,m)}\phi_n\quad \text{in}\quad \Om,
\\
\psi_1^{(n,m)}=&0\quad \text{on}\quad \p\Om.
\end{aligned}
\end{equation}
Employing the eigenvalue equation for $\phi_n$, by direct
calculations we check that
\begin{equation}\label{3.9b}
(-\D_\xi-\l_n)\mathrm{q}\phi_n=-2 \left(
\k_1\frac{\p}{\p\xi_2}-\k_2\frac{\p}{\p\xi_3}\right)\phi_n.
\end{equation}
Hence, the problem (\ref{3.12}) is solvable for
\begin{equation}\label{3.9a}
\l_{-1}^{(n,m)}=0
\end{equation}
with a solution given by the identity
\begin{equation}\label{3.13}
\psi_1^{(n,m)}(s,\xi)=\frac{1}{2}\Psi_0^{(n,m)}(s)\phi_n(\xi)\mathrm{q}(s,\xi)+
\Psi_1^{(n,m)}(s)\phi_n(\xi),
\end{equation}
where the function $\Psi_1^{(n,m)}$ is unknown and should satisfy
the boundary conditions
\begin{equation*}
\Psi_1^{(n,m)}(0)=\Psi_1^{(n,m)}(s_0)=0.
\end{equation*}
The above obtained formula for $\l_{-1}^{(n,m)}$ and that for
$\l_0^{(n,m)}$ in (\ref{3.8}) prove the formulas for the first terms
in (\ref{1.1}).

We substitute the formulas (\ref{3.5a}), (\ref{3.7}), (\ref{3.8}),
(\ref{3.8a}), (\ref{3.9a})  into the problem (\ref{3.5}) for $i=2$
to obtain
\begin{align}
&
\begin{aligned}
(-\D_\xi-\l_n)\psi_2^{(n,m)}&=
\l_0^{(n,m)}\psi_0^{(n,m)}+F_2^{(n,m)} \quad \text{in}\quad \Om,
\\
 \psi_2^{(n,m)}&= 0\quad \text{on}\quad \p\Om,
\end{aligned}
\label{3.15}
\\
&F_2^{(n,m)}=\frac{1}{2}(\mathcal{F}_1-\l_n\mathrm{q})\mathrm{q}\psi_0^{(n,m)}
+\mathcal{F}_2\psi_0^{(n,m)}- \Psi_1^{(n,m)} \left(
\k_1\frac{\p}{\p\xi_2}-\k_2\frac{\p}{\p\xi_3}\right)\phi_n.\label{3.12b}
\end{align}
Taking into account the eigenvalue equation for $\phi_n$, by direct
calculations we check
\begin{equation}\label{3.16a}
(\mathcal{F}_1-\l_n\mathrm{q})\mathrm{q}\phi_n=-3\mathrm{q} \left(
\k_1\frac{\p}{\p\xi_2}-\k_2\frac{\p}{\p\xi_3}\right)\phi_n -
(\k_1^2+\k_2^2)\phi_n.
\end{equation}
Hence, we can rewrite the formula for $F_2^{(n,m)}$ as follows,
\begin{align}
&F_2^{(n,m)}=\widetilde{\mathcal{F}}\psi_0^{(n,m)}-\Psi_1^{(n,m)}
\left(\k_1\frac{\p}{\p\xi_2}-\k_2\frac{\p}{\p\xi_3}
\right)\phi_n,\label{3.10}
\\
&\widetilde{\mathcal{F}}\psi_0^{(n,m)}=\frac{1}{2}(\mathcal{F}_1-
\l_n\mathrm{q})\mathrm{q}\psi_0^{(n,m)}
+\mathcal{F}_2\psi_0^{(n,m)},\label{3.18c}
\\
&\widetilde{\mathcal{F}}:=-\frac{3\mathrm{q}}{2}
\left(\k_1\frac{\p}{\p\xi_2}-\k_2\frac{\p}{\p\xi_3}\right)-
\frac{\k_1^2+\k_2^2}{2}+\frac{\p^2}{\p s^2}+\left(\k_3\frac{\p}{\p
s}+\frac{\p}{\p s}\k_3\right)\mathcal{R}
+\k_3^2\mathcal{R}^2.\label{3.12a}
\end{align}
In (\ref{3.15}) the Laplace operator is taken only with respect to
$\xi$, and this problem involves $s$ as a parameter. So, we can
consider (\ref{3.15}) as a problem for the Dirichlet Laplacian
$\mathcal{S}$ in $\om$ with a dependence on $s$. Since $\l_n$ is a
simple eigenvalue of $\mathcal{S}$, the solvability condition of
(\ref{3.15}) is the orthogonality of the right hand side in the
equation to $\phi_n$ in $L_2(\om)$,

\begin{equation}\label{3.16}
\l_0^{(n,m)}\Psi_0^{(n,m)}+\big(F_2^{(n,m)},\phi_n\big)_{L_2(\om)}=0,\quad
s\in(0,s_0),
\end{equation}
where we have taken into account the normalization of $\phi_n$ and
the formula (\ref{3.8}). Let us evaluate the second term in the left
hand side of (\ref{3.16}).

Integrating by parts, we obtain
\begin{align}
&
\begin{aligned}
\int\limits_{\om} \phi_n \mathrm{q}
&\left(\k_1\frac{\p}{\p\xi_2}-\k_2\frac{\p}{\p\xi_3}\right)
\phi_n\di\xi= \frac{1}{2}\int\limits_{\om} \mathrm{q}
\left(\k_1\frac{\p}{\p\xi_2}-\k_2\frac{\p}{\p\xi_3}\right)
\phi_n^2\di\xi
\\
&=-\frac{1}{2}\int\limits_{\om} \phi_n^2
\left(\k_1\frac{\p}{\p\xi_2}-\k_2\frac{\p}{\p\xi_3}\right)
\mathrm{q}\di\xi= -\frac{\k_1^2+\k_2^2}{2},
\end{aligned}
\nonumber
\\
&
\begin{aligned}
\int\limits_{\om} \phi_n \mathcal{R} \phi_n\di\xi=
\frac{1}{2}\int\limits_{\om} \mathcal{R}\phi_n^2\di\xi=0,
\end{aligned}
\nonumber
\\
&
\begin{aligned}
\int\limits_{\om} \phi_n\mathcal{R}^2 \phi_n\di\xi=
-\int\limits_{\om} |\mathcal{R} \phi_n|^2\di\xi =-C_n(\om),
\end{aligned}\nonumber
\\
&
\begin{aligned}
\int\limits_{\om} \phi_n
\left(\k_1\frac{\p}{\p\xi_2}-\k_2\frac{\p}{\p\xi_3}\right)\phi_n\di\xi=
\frac{1}{2}\int\limits_{\om}
\left(\k_1\frac{\p}{\p\xi_2}-\k_2\frac{\p}{\p\xi_3}\right)\phi_n^2\di\xi
=0.
\end{aligned}
\label{3.13c}
\end{align}
We substitute the identities obtained, (\ref{3.8}), (\ref{3.10}),
(\ref{3.12a}) into (\ref{3.16}) and arrive at the equation
\begin{equation*}
\frac{\p^2\Psi_0^{(n,m)}}{\p s^2}+
\left(\frac{\k_1^2+\k_2^2}{4}-\k_3^2C_n(\om)\right)\Psi_0^{(n,m)}+
\l_0^{(n,m)}\Psi_0^{(n,m)}=0,\quad s\in(0,s_0).
\end{equation*}
Together with the boundary condition (\ref{3.9}) it can be rewritten
as
\begin{equation}\label{3.22}
\mathcal{L}_n\Psi_0^{(n,m)}=\l_0^{(n,m)}\Psi_0^{(n,m)}.
\end{equation}
Thus, $\l_0^{(n,m)}$ is a (simple) eigenvalue of $\mathcal{L}_n$ and
$\Psi_0^{(n,m)}$ is the associated eigenfunction. In what follows we
assume the eigenfunctions $\Psi_0^{(n,m)}$ are assumed to be
orthonormalized in $L_2(0,s_0)$. We also note that by the smoothness
improving theorems $\Psi_0^{(n,m)}\in C^\infty[0,s_0]$.

Let $V_n$ be the orthogonal complement to $\{\phi_n\}$ in
$L_2(\om)$. By $\mathcal{S}_n^\bot$ we denote the restriction of
$\mathcal{S}$ on $V_n\cap\Wo^2(\om)$. It is clear that the operator
$(\mathcal{S}_n^\bot-\l)^{-1}$ is well-defined in and bounded as
that from $V_n$ in $\Wo^2(\om)$.

It follows from (\ref{3.13c}) that
\begin{equation}\label{3.13b}
\left(\k_1\frac{\p}{\p\xi_2}-\k_2\frac{\p}{\p\xi_3}\right)\phi_n\in
V_n.
\end{equation}
The identity (\ref{3.16}) is satisfied due to (\ref{3.22}) and it
yields $\l_0^{(n,m)}\psi_0^{(n,m)}+F_2^{(n,m)}\in V_n$. Hence, by
(\ref{3.10}), (\ref{3.13b}) we have
$(\widetilde{\mathcal{F}}+\l_0^{(n,m)})\psi_0^{(n,m)}\in V_n$.
Taking into account this fact, (\ref{3.10}), and (\ref{3.9b}), we
return back to the problem (\ref{3.15}), and write its solution as
\begin{align}
& \psi_2^{(n,m)}(s,\xi)=\widetilde{\psi}_2^{(n,m)}(s,\xi)
+\frac{1}{2}\Psi_1^{(n,m)}(s) \phi_n(\xi)\mathrm{q}(s,\xi)
+\Psi_2^{(n,m)}(s)\phi_n(\xi), \label{3.23}
\\
&\widetilde{\psi}_2^{(n,m)}:=(\mathcal{S}_n^\bot-\l_n)^{-1}
(\widetilde{\mathcal{F}}+\l_0^{(n,m)})\psi_0^{(n,m)},\nonumber
\end{align}
where the function $\Psi_2^{(n,m)}$ is unknown and should satisfy
the boundary conditions
\begin{equation*}
\Psi_2^{(n,m)}(0)=\Psi_2^{(n,m)}(s_0)=0.
\end{equation*}
Bearing in mind the belongings $\phi_n\in C^\infty(\overline{\om})$,
$\k_i\in C^\infty[0,s_0]$, and the identities (\ref{3.12a}),
(\ref{1.0b}), it is easy to see that the function
$(\widetilde{\mathcal{F}}+\l_0^{(n,m)})\psi_0^{(n,m)}$ can be
represented as a finite sum
\begin{equation*}
(\widetilde{\mathcal{F}}+\l_0^{(n,m)})\psi_0^{(n,m)}=\sum\limits_{j}
\Psi_{2,j}^{(n,m)}(s)F_{2,j}^{(n,m)}(\xi),
\end{equation*}
where $\Psi_{2,j}^{(n,m)}\in C^\infty[0,s_0]$, $F_{2,j}^{(n,m)}\in
C^\infty(\overline{\om})\cap V_n$. Thus,
\begin{equation*}
\widetilde{\psi}_2^{(n,m)}(s,\xi)=\sum\limits_{j}
\Psi_{2,j}^{(n,m)}(s) \phi_{2,j}^{(n,m)}(\xi),\quad
\phi_{2,j}^{(n,m)}=(\mathcal{S}_n^\bot-\l_n)^{-1}F_{2,j}^{(n,m)}.
\end{equation*}
By the smoothness improving theorems $\phi_{2,j}^{(n,m)}\in
C^\infty(\overline{\om})$, and therefore
$\widetilde{\psi}_2^{(n,m)}\in C^\infty(\overline{\Om})$.

We substitute the formulas (\ref{3.5a}), (\ref{3.7}), (\ref{3.8}),
(\ref{3.8a}), (\ref{3.9a}), (\ref{3.13}), (\ref{3.16a}),
(\ref{3.12a}), (\ref{3.23}) into the problem (\ref{3.5}) for $i=3$,
\begin{align}
&
\begin{aligned}
(-\D_\xi-\l_n)\psi_3^{(n,m)}&=\l_0^{(n,m)}\psi_1^{(n,m)}+
\l_1^{(n,m)}\psi_0^{(n,m)}+F_3^{(n,m)} \quad \text{in}\quad \Om,
\\
\psi_3^{(n,m)}&=0\quad \text{on}\quad \p\Om,
\end{aligned}
\label{3.25}
\\
&
\begin{aligned}
F_3^{(n,m)}=&\widetilde{F}_3^{(n,m)}+\widetilde{\mathcal{F}}\Psi_1^{(n,m)}\phi_n-
\Psi_2^{(n,m)}\left(\k_1\frac{\p}{\p\xi_2}-\k_2\frac{\p}{\p\xi_3}
\right)\phi_n,
\\
\widetilde{F}_3^{(n,m)}:=
&(\mathcal{F}_1-\l_n\mathrm{q})\widetilde{\psi}_2^{(n,m)}+\frac{1}{2}
\mathcal{F}_2\mathrm{q}\psi_0^{(n,m)}+(\mathcal{F}_3-\l_0^{(n,m)}\mathrm{q})
\psi_0^{(n,m)}.
\end{aligned}\label{3.26}
\end{align}
We again treat this problem as that for $\mathcal{S}$ and depending
on $s$, and the corresponding solvability condition is
\begin{equation}\label{3.27}
\l_0^{(n,m)}(\psi_1^{(n,m)},\phi_n)_{L_2(\om)}+\l_1^{(n,m)}
\Psi_0^{(n,m)}+(F_3^{(n,m)},\phi_n)_{L_2(\om)}=0.
\end{equation}
In the same way how the equation (\ref{3.22}) was derived we obtain
\begin{align}
&(\mathcal{L}_n-\l_0^{(n,m)})\Psi_1^{(n,m)}
=\l_1^{(n,m)}\Psi_0^{(n,m)}+f_3^{(n,m)},\label{3.28}
\\
&f_3^{(n,m)}(s):=(\widetilde{F}_3^{(n,m)}(s,\cdot),\phi_n)_{L_2(\om)}+\frac{1}{2}
\l_0^{(n,m)}\Psi_0^{(n,m)}(s)\mathrm{q}_n(s),\nonumber
\\
&\mathrm{q}_n(s):=(\mathrm{q}(s,\cdot)\phi_n,\phi_n)_{L_2(\om)}=\k_1(s)
(\xi_2\phi_n,\phi_n)_{L_2(\om)}-\k_2(s)(\xi_3\phi_n,\phi_n)_{L_2(\om)}.
\nonumber
\end{align}
Since $\l_0^{(n,m)}$ is an eigenvalue of $\mathcal{L}_n$, the
solvability condition of the last equation is the orthogonality in
$L_2(0,s_0)$ of its right hand side to the eigenfunctions associated
with $\l_0^{(n,m)}$, i.e., it should be orthogonal to
$\Psi_0^{(n,m)}$. It gives the formula for $\l_1^{(n,m)}$,
\begin{equation}\label{3.29}
\begin{aligned}
\l_1^{(n,m)}&=-(f_3^{(n,m)},\Psi_0^{(n,m)})_{L_2(0,s_0)}
\\
&=
-(\widetilde{F}_3^{(n,m)},\psi_0^{(n,m)})_{L_2(\Om)}-\frac{1}{2}\l_0^{(n,m)}
(\psi_0^{(n,m)},\mathrm{q}\psi_0^{(n,m)})_{L_2(\Om)}.
\end{aligned}
\end{equation}
Let us calculate the right hand side of this identity. Integrating
by parts and employing (\ref{3.8a}), (\ref{3.9b}), (\ref{3.12b}),
(\ref{3.18c}) we have
\begin{equation}
\begin{aligned}
\big((\mathcal{F}_1-&\mathrm{q}\l_n)\widetilde{\psi}_2^{(n,m)},
\psi_0^{(n,m)}\big)_{L_2(\Om)} =\big(\widetilde{\psi}_2^{(n,m)},
(\mathcal{F}_1-\mathrm{q}\l_n)\psi_0^{(n,m)}\big)_{L_2(\Om)}
\\
&=-\frac{1}{2} \big(\widetilde{\psi}_2^{(n,m)},
(\D_\xi+\l_n)\mathrm{q}\psi_0^{(n,m)}\big)_{L_2(\Om)}
\\
&=-\frac{1}{2} \big((\D_\xi+\l_n)\widetilde{\psi}_2^{(n,m)},
\mathrm{q}\psi_0^{(n,m)}\big)_{L_2(\Om)}
\\
&=\frac{1}{2} \big(
(\widetilde{\mathcal{F}}+\l_0^{(n,m)})\psi_0^{(n,m)},\mathrm{q}\psi_0^{(n,m)}
\big)_{L_2(\Om)}
\\
&=\frac{1}{2}
\left(\frac{1}{2}(\mathcal{F}_1-\mathrm{q}\l_n)\mathrm{q}\psi_0^{(n,m)}
+(\mathcal{F}_2+\l_0^{(n,m)})\psi_0^{(n,m)},
\mathrm{q}\psi_0^{(n,m)}\right)_{L_2(\Om)}.
\end{aligned}\label{3.20a}
\end{equation}
Employing the eigenvalue equation for $\phi_n$, by direct
calculations we check
\begin{equation*}
(\mathcal{F}_1-\mathrm{q}\l_n)\mathrm{q}\psi_0^{(n,m)}=-3\mathrm{q}
\left(\k_1\frac{\p}{\p\xi_2}-\k_2\frac{\p}{\p\xi_3}
\right)\psi_0^{(n,m)}-(\k_1^2+\k_2^2)\psi_0^{(n,m)}.
\end{equation*}
We again integrate by parts,
\begin{align*}
\big( (\mathcal{F}_1-\mathrm{q}\l_n)\mathrm{q}\psi_0^{(n,m)},
\mathrm{q}\psi_0^{(n,m)}\big)_{L_2(\Om)} =&-3 \left(\mathrm{q}^2
\left(\k_1\frac{\p}{\p\xi_2} -\k_2
\frac{\p}{\p\xi_3}\right)\psi_0^{(n,m)},
\psi_0^{(n,m)}\right)_{L_2(\Om)}
\\
& -\big((\k_1^2+\k_2^2)\psi_0^{(n,m)}, \mathrm{q}\psi_0^{(n,m)}
\big)_{L_2(\Om)}
\\
=&-\frac{3}{2} \left(\mathrm{q}^2, \left(\k_1\frac{\p}{\p\xi_2}-\k_2
\frac{\p}{\p\xi_3}\right)\big(\psi_0^{(n,m)}\big)^2
\right)_{L_2(\Om)}
\\
 &-\big((\k_1^2+\k_2^2)\psi_0^{(n,m)},
\mathrm{q}\psi_0^{(n,m)} \big)_{L_2(\Om)}
\\
=&2\big((\k_1^2+\k_2^2)\psi_0^{(n,m)}, \mathrm{q}\psi_0^{(n,m)}
\big)_{L_2(\Om)}.
\end{align*}
Substituting the identities obtained into (\ref{3.20a}), we arrive
at
\begin{equation}
\big((\mathcal{F}_1-\mathrm{q}\l_n)\widetilde{\psi}_2^{(n,m)},
\psi_0^{(n,m)}\big)_{L_2(\Om)}=\frac{1}{2} \left(
(\mathcal{F}_2+\k_1^2+\k_2^2+\l_0^{(n,m)})\psi_0^{(n,m)},\mathrm{q}\psi_0^{(n,m)}
\right)_{L_2(\Om)}. \label{3.20c}
\end{equation}
It follows from (\ref{3.7}) that
\begin{equation*}
\mathcal{F}_3=\mathrm{q}\mathcal{F}_2+\frac{\p \mathrm{q}}{\p
s}\frac{\p}{\p s}+(\mathcal{R}\mathrm{q}) \k_3\frac{\p}{\p
s}+\frac{\p \mathrm{q}}{\p s} \k_3 \mathcal{R}+
\k_3^2(\mathcal{R}\mathrm{q})\mathcal{R}.
\end{equation*}
We substitute this identity and (\ref{3.20c}) into (\ref{3.29}) and
integrate by parts,
\begin{equation}\label{3.21a}
\begin{aligned}
\l_1^{(n,m)}=& -\frac{1}{2} \left(
(\mathcal{F}_2+\k_1^2+\k_2^2)\psi_0^{(n,m)},
\mathrm{q}\psi_0^{(n,m)} \right)_{L_2(\Om)}
\\
&-\frac{1}{2} \left(\mathcal{F}_2
\mathrm{q}\psi_0^{(n,m)},\psi_0^{(n,m)} \right)_{L_2(\Om)}
-\left(\mathrm{q}\mathcal{F}_2 \psi_0^{(n,m)},\psi_0^{(n,m)}
\right)_{L_2(\Om)}
\\
&- \left(\frac{\p\mathrm{q}}{\p s}\frac{\p\psi_0^{(n,m)}}{\p s},
\psi_0^{(n,m)}\right)_{L_2(\Om)} -
\left(\k_3\frac{\p\psi_0^{(n,m)}}{\p s}\mathcal{R}\mathrm{q},
\psi_0^{(n,m)}\right)_{L_2(\Om)}
\\
&-\left(\k_3\frac{\p\mathrm{q}}{\p s}\mathcal{R}\psi_0^{(n,m)},
\psi_0^{(n,m)}\right)_{L_2(\Om)}
-\left(\k_3^2\mathcal{R}\psi_0^{(n,m)},
\psi_0^{(n,m)}\mathcal{R}\mathrm{q}\right)_{L_2(\Om)}
\\
=&-2\big(\mathcal{F}_2\psi_0^{(n,m)},
\mathrm{q}\psi_0^{(n,m)}\big)_{L_2(\Om)}  -\frac{1}{2}
\big((\k_1^2+\k_2^2)\psi_0^{(n,m)},\mathrm{q}\psi_0^{(n,m)}\big)_{L_2(\Om)}
\\
&-\left(\frac{\p\mathrm{q}}{\p s}\frac{\p\psi_0^{(n,m)}}{\p s},
\psi_0^{(n,m)}\right)_{L_2(\Om)} -
\left(\k_3\frac{\p\psi_0^{(n,m)}}{\p s}\mathcal{R}\mathrm{q},
\psi_0^{(n,m)}\right)_{L_2(\Om)}
\\
&-\left(\k_3\frac{\p\mathrm{q}}{\p s}\mathcal{R}\psi_0^{(n,m)},
\psi_0^{(n,m)}\right)_{L_2(\Om)}
-\left(\k_3^2\mathcal{R}\psi_0^{(n,m)},
\psi_0^{(n,m)}\mathcal{R}\mathrm{q}\right)_{L_2(\Om)}.
\end{aligned}
\end{equation}
In view of (\ref{3.5a}), (\ref{3.22}) we have
\begin{equation}\label{3.32}
\begin{aligned}
-\frac{1}{2}
&\big((\k_1^2+\k_2^2)\psi_0^{(n,m)},\mathrm{q}\psi_0^{(n,m)}\big)_{L_2(\Om)}
-2\big(\mathcal{F}_2\psi_0^{(n,m)},
\mathrm{q}\psi_0^{(n,m)}\big)_{L_2(\Om)}
\\
=& 2\big( (\l_0^{(n,m)}-C_n(\om)\k_3^2)
\psi_0^{(n,m)},\mathrm{q}\psi_0^{(n,m)} \big)_{L_2(\Om)} -2
\left(\mathcal{R}\k_3  \frac{\p\psi_0^{(n,m)}}{\p s},\mathrm{q}
\psi_0^{(n,m)}\right)_{L_2(\Om)}
\\
&-2 \left(\frac{\p}{\p s}\k_3 \mathcal{R}\psi_0^{(n,m)},\mathrm{q}
\psi_0^{(n,m)}\right)_{L_2(\Om)} -2 \big(
\k_3^2\mathcal{R}^2\psi_0^{(n,m)},\mathrm{q}\psi_0^{(n,m)}
\big)_{L_2(\Om)}.
\end{aligned}
\end{equation}
We integrate by parts employing (\ref{3.8}),
\begin{equation}\label{3.33}
\begin{aligned}
-&2 \left(\mathcal{R}\k_3  \frac{\p\psi_0^{(n,m)}}{\p s},\mathrm{q}
\psi_0^{(n,m)}\right)_{L_2(\Om)}-2 \left(\frac{\p}{\p s}\k_3
\mathcal{R}\psi_0^{(n,m)},\mathrm{q}
\psi_0^{(n,m)}\right)_{L_2(\Om)}
\\
-& \left(\k_3\frac{\p\psi_0^{(n,m)}}{\p s}\mathcal{R}\mathrm{q},
\psi_0^{(n,m)}\right)_{L_2(\Om)} -\left(\k_3\frac{\p\mathrm{q}}{\p
s}\mathcal{R}\psi_0^{(n,m)}, \psi_0^{(n,m)}\right)_{L_2(\Om)}
\\
=& 2 \left(\k_3 \frac{\p\psi_0^{(n,m)}}{\p s},\mathcal{R}\mathrm{q}
\psi_0^{(n,m)}\right)_{L_2(\Om)}+2 \left(\k_3
\mathcal{R}\psi_0^{(n,m)},\frac{\p}{\p s}\mathrm{q}
\psi_0^{(n,m)}\right)_{L_2(\Om)}
\\
&- \left(\k_3\frac{\p\psi_0^{(n,m)}}{\p s}\mathcal{R}\mathrm{q},
\psi_0^{(n,m)}\right)_{L_2(\Om)} -\left(\k_3\frac{\p\mathrm{q}}{\p
s}\mathcal{R}\psi_0^{(n,m)}, \psi_0^{(n,m)}\right)_{L_2(\Om)}
\\
= &\left(\k_3\frac{\p\psi_0^{(n,m)}}{\p s},
\psi_0^{(n,m)}\mathcal{R}\mathrm{q}\right)_{L_2(\Om)}+
\left(\k_3\frac{\p\mathrm{q}}{\p s}\mathcal{R}\psi_0^{(n,m)},
\psi_0^{(n,m)}\right)_{L_2(\Om)}
\\
&+4
\left(\k_3\mathrm{q}\mathcal{R}\psi_0^{(n,m)},\frac{\p\psi_0^{(n,m)}}{\p
s}\right)_{L_2(\Om)}
\\
=&\frac{1}{2}\left(\k_3 \mathcal{R}\mathrm{q},\frac{\p}{\p s}
\big(\psi_0^{(n,m)}\big)^2\right)_{L_2(\Om)} + \frac{1}{2}
\left(\k_3\frac{\p\mathrm{q}}{\p
s},\mathcal{R}\big(\psi_0^{(n,m)}\big)^2 \right)_{L_2(\Om)}
\\
&+ \left( \k_3\mathrm{q}\mathcal{R}\phi_n^2,\frac{\p}{\p s}
\big(\Psi_0^{(n,m)}\big)^2 \right)_{L_2(\Om)}
\\
=&-\frac{1}{2}\left(\frac{\p}{\p s}\k_3 \mathcal{R}\mathrm{q},
\big(\psi_0^{(n,m)}\big)^2\right)_{L_2(\Om)} - \frac{1}{2}
\left(\mathcal{R}\k_3\frac{\p\mathrm{q}}{\p
s},\big(\psi_0^{(n,m)}\big)^2 \right)_{L_2(\Om)}
\\
&+ \left( \phi_n^2 \big(\Psi_0^{(n,m)}\big)^2, \frac{\p}{\p
s}\k_3\mathcal{R}\mathrm{q}\right)_{L_2(\Om)}
\\
 = &\frac{1}{2} \left(\psi_0^{(n,m)},\psi_0^{(n,m)}\frac{\p}{\p s}
\k_3\mathcal{R}\mathrm{q}\right)_{L_2(\Om)}-\frac{1}{2}
\left(\psi_0^{(n,m)},\psi_0^{(n,m)}\k_3
\mathcal{R}\frac{\p\mathrm{q}}{\p s}\right)_{L_2(\Om)}
\\
= &\frac{1}{2} \left(\psi_0^{(n,m)},\psi_0^{(n,m)}\k'_3
\mathcal{R}\mathrm{q}\right)_{L_2(\Om)}.
\end{aligned}
\end{equation}
In the same fashion we obtain
\begin{equation}\label{3.36a}
-\left(\frac{\p\mathrm{q}}{\p s}\frac{\p\psi_0^{(n,m)}}{\p s},
\psi_0^{(n,m)}\right)_{L_2(\Om)}=\frac{1}{2} \left(\psi_0^{(n,m)},
\frac{\p^2\mathrm{q}}{\p s^2} \psi_0^{(n,m)}\right)_{L_2(\Om)}.
\end{equation}
Employing the identity $\mathcal{R}^2\mathrm{q}=-\mathrm{q}$, we
have
\begin{align*}
-&\left(\k_3^2\mathcal{R}\psi_0^{(n,m)},
\psi_0^{(n,m)}\mathcal{R}\mathrm{q}\right)_{L_2(\Om)} -2 \big(
\k_3^2\mathcal{R}^2\psi_0^{(n,m)},\mathrm{q}\psi_0^{(n,m)}
\big)_{L_2(\Om)}
\\
& = -\left(\k_3^2\mathcal{R}\psi_0^{(n,m)},
\psi_0^{(n,m)}\mathcal{R}\mathrm{q}\right)_{L_2(\Om)} +2 \big(
\k_3^2\mathcal{R}\psi_0^{(n,m)},\mathcal{R}\mathrm{q}\psi_0^{(n,m)}
\big)_{L_2(\Om)}
\\
&=2 \big(
\k_3^2\mathcal{R}\psi_0^{(n,m)},\mathrm{q}\mathcal{R}\psi_0^{(n,m)}
\big)_{L_2(\Om)}+\big( \k_3^2\mathcal{R}\psi_0^{(n,m)},
\psi_0^{(n,m)}\mathcal{R}\mathrm{q} \big)_{L_2(\Om)}
\\
&=2 \big(
\k_3^2\mathcal{R}\psi_0^{(n,m)},\mathrm{q}\mathcal{R}\psi_0^{(n,m)}
\big)_{L_2(\Om)}+ \frac{1}{2} \big(
\k_3^2\mathcal{R}(\psi_0^{(n,m)})^2, \mathcal{R}\mathrm{q}
\big)_{L_2(\Om)}
\\
&=2 \big(
\k_3^2\mathcal{R}\psi_0^{(n,m)},\mathrm{q}\mathcal{R}\psi_0^{(n,m)}
\big)_{L_2(\Om)}-\frac{1}{2} \big(
\k_3^2\psi_0^{(n,m)},\psi_0^{(n,m)}\mathcal{R}^2\mathrm{q}
\big)_{L_2(\Om)}
\\
&= 2 \big(
\k_3^2\mathcal{R}\psi_0^{(n,m)},\mathrm{q}\mathcal{R}\psi_0^{(n,m)}
\big)_{L_2(\Om)}+\frac{1}{2}(\k_3^2\psi_0^{(n,m)},
\mathrm{q}\psi_0^{(n,m)})_{L_2(\Om)}.
\end{align*}
We substitute the identities obtained and (\ref{3.32}),
(\ref{3.33}), (\ref{3.36a}) into (\ref{3.21a}) and arrive at the
formula (\ref{1.2}) for $\l_1^{(n,m)}$.

Let $V_{n,m}$ be the orthogonal complement to $\{\Psi_0^{(n,m)}\}$
in $L_2(0,s_0)$, and $\mathcal{L}_{n,m}^\bot$ be the restriction of
$\mathcal{L}_n$ to $V_{n,m}\cap \Wo^2(0,s_0)$. The operator
$(\mathcal{L}_{n,m}^\bot-\l_0^{(n,m)})^{-1}$ is well-defined and
bounded as that from $V_{n,m}$ into $V_{n,m}\cap\Wo^2(0,s_0)$.

The orthogonality condition (\ref{3.29}) means that the right hand
side of (\ref{3.28}) is orthogonal to $\Psi_0^{(n,m)}$. Thus,
$\l_1^{(n,m)}\Psi_0^{(n,m)}+f_3^{(n,m)}\in V_{n,m}$. In view of this
fact we can choose a solution to (\ref{3.28}) as
\begin{equation*}
\Psi_1^{(n,m)}=(\mathcal{L}_{n,m}^\bot-\l_0^{(n,m)})^{-1}
(\l_1^{(n,m)}\Psi_0^{(n,m)}+f_3^{(n,m)}).
\end{equation*}
By the smoothness improving theorems $\Psi_1^{(n,m)}\in
C^\infty[0,s_0]$.

The equations (\ref{3.27}), (\ref{3.28}) being satisfied, the right
hand side of the equation in (\ref{3.25}) is orthogonal to $\phi_n$,
i.e.,
\begin{equation*}
\l_0^{(n,m)}\psi_1^{(n,m)}+\l_1^{(n,m)}\psi_0^{(n,m)}+F_3^{(n,m)}\in
V_n.
\end{equation*}
By (\ref{3.13b}) it yields
\begin{equation*}
\l_0^{(n,m)}\psi_1^{(n,m)}+\l_1^{(n,m)}\psi_0^{(n,m)}+\widetilde{F}_3^{(n,m)}
+\widetilde{\mathcal{F}}\Psi_1^{(n,m)}\phi_n \in V_n.
\end{equation*}
In view of this fact we can choose a solution to (\ref{3.25})
as
\begin{align*}
\psi_3^{(n,m)}(s,\xi)=&\widetilde{\psi}_3^{(n,m)}(s,\xi)
+\frac{1}{2}\Psi_2^{(n,m)}(s) \phi_n(\xi)\mathrm{q}(s,\xi)
+\Psi_3^{(n,m)}(s)\phi_n(\xi),
\\
\widetilde{\psi}_3^{(n,m)}:=&(\mathcal{S}_n^\bot-\l_n)^{-1}
\big(\l_0^{(n,m)}\psi_1^{(n,m)}+ \l_1^{(n,m)}\psi_0^{(n,m)}
+\widetilde{F}_3^{(n,m)}+\widetilde{\mathcal{F}}
\Psi_1^{(n,m)}\phi_n\big),
\end{align*}
where the function $\Psi_3^{(n,m)}$ is unknown  and should satisfy
the boundary conditions
\begin{equation*}
\Psi_3^{(n,m)}(0)=\Psi_3^{(n,m)}(s_0)=0,
\end{equation*}
and $\widetilde{\psi}_3^{(n,m)}\in C^\infty(\overline{\Om})$. This
smoothness is proved in the same way as for
$\widetilde{\psi}_2^{(n,m)}$.

The remaining problems for $\psi_i$, $i\geqslant 4$, are solved in
the same way as for $i=1,2,3$. Namely, the solvability condition of
these problems is the orthogonality of the right hand side to
$\phi_n$ in $L_2(\om)$ for each $s\in(0,s_0)$. In its turn, these
condition imply the problems for $\Psi_i$. The solvability
conditions of these problems imply the formulas for $\l_i$. The
result of this recurrent procedure is formulated in

\begin{lmm}\label{lm3.1}
There exist solutions to the problems (\ref{3.5}) given by the
identities
\begin{equation}\label{3.23b}
\psi_i^{(n,m)}(s,\xi)=\widetilde{\psi}_i^{(n,m)}(s,\xi)
+\frac{1}{2}\Psi_{i-1}^{(n,m)}(s)\mathrm{q}(s,\xi)\phi_n(\xi)
+\Psi_i^{(n,m)}(s)\phi_n(\xi),\quad i\geqslant 0.
\end{equation}
The functions $\widetilde{\psi}_i^{(n,m)}\in
C^\infty(\overline{\Om})$, $\Psi_i^{(n,m)}\in C^\infty[0,s_0]$ read
as follows,
\begin{align}
& \widetilde{\psi}_i^{(n,m)}(s,\xi)=0,\quad i\leqslant 1, \nonumber
\\
&
\begin{aligned}
\widetilde{\psi}_i^{(n,m)}(s,\xi)=(\mathcal{S}_n^\bot-\l_n)^{-1}
\Bigg(&\widetilde{F}_i^{(n,m)}+\widetilde{\mathcal{F}} \Psi_{i-2}^{(n,m)}\phi_n
+\sum\limits_{j=2}^{i} \l_{j-2}^{(n,m)}\psi_{i-j}^{(n,m)}\Bigg),\quad i\geqslant
2,
\end{aligned}\nonumber
\\
&\widetilde{F}_i^{(n,m)}=0,\quad i\leqslant 2,\nonumber
\\
&
\begin{aligned}
\widetilde{F}_i^{(n,m)}=&(\mathcal{F}_1-\l_n\mathrm{q})
\widetilde{\psi}_{i-1}^{(n,m)} + \mathcal{F}_2 \left(
\widetilde{\psi}_{i-2}^{(n,m)} + \frac{1}{2} \Psi_{i-3}^{(n,m)}\mathrm{q}\phi_n
\right)
\\
&+ \sum\limits_{j=3}^{i}
(\mathcal{F}_j-\l_{j-3}^{(n,m)}\mathrm{q})\psi_{i-j}^{(n,m)},\quad
i\geqslant 3,
\end{aligned}\label{3.23a}
\\
&\Psi^{(n,m)}_i=(\mathcal{L}_{n,m}^\bot-\l_0^{(n,m)})^{-1} \left(
f_{i+2}^{(n,m)}  +\sum\limits_{j=1}^{i} \l_j^{(n,m)}
\Psi_{i-j}^{(n,m)}\right),\quad i\geqslant 1,\nonumber
\\
& f_i^{(n,m)}=0,\quad i\leqslant 2,
\qquad
 f_i^{(n,m)}=(\widetilde{F}_i,\psi_0^{(n,m)})_{L_2(\Om)} +
\frac{1}{2}\sum\limits_{j=0}^{i-3}\l_j^{(n,m)} \Psi_{i-j-3}^{(n,m)}
\mathrm{q}_n,\quad i\geqslant 3,\nonumber
\end{align}
where $\l_{-2}^{(n,m)}$ and $\l_{-1}^{(n,m)}$ are determined by (\ref{3.8}),
(\ref{3.9a}), and
\begin{equation}\label{3.42}
\l_i^{(n,m)}=-(\widetilde{F}_{i+2}^{(n,m)},\psi_0^{(n,m)})_{L_2(\Om)}
-\frac{1}{2} \sum\limits_{j=0}^{i-1} \l_j^{(n,m)}
\big(\Psi_{i-j-1}^{(n,m)}\phi_n,\mathrm{q}\psi_0^{(n,m)}\big)_{L_2(\Om)},\quad
i\geqslant 1.
\end{equation}
\end{lmm}

\begin{rmrk}\label{rm3.1}
The formula (\ref{3.23a}) involves the term
$\widetilde{\psi}_{i-2}^{(n,m)}$. It is zero for $i\geqslant 3$, and
this is why it is absent in (\ref{3.10}), (\ref{3.26}). At the same
time, it is easy to see that it comes from the term
$(\mathcal{F}_2-\l_{j-3}^{(n,m)})\psi_{i-2}^{(n,m)}$ in
(\ref{3.5a}), when $\psi_{i-2}^{(n,m)}$ is taken in accordance with
(\ref{3.23b}).
\end{rmrk}

Given  any $N\geqslant 1$, we denote
\begin{equation}\label{3.24a}
\l_{\e,N}^{(n,m)}:=\e^{-2}\l_n+\sum\limits_{i=0}^{N-2}\e^i\l_i^{(n,m)},\quad
\psi_{\e,N}^{(n,m)}(s,\xi):=\sum\limits_{i=0}^{N}\e^i\psi_i^{(n,m)}(s,\xi).
\end{equation}

The next lemma follows directly from Lemma~\ref{lm3.1}.

\begin{lmm}\label{lm3.2}
The function $\psi_{\e,N}^{(n,m)}\in C^\infty(\overline{\Om})$ and
the number $\l_{\e,N}^{(n,m)}$ satisfy the equation
\begin{equation}\label{3.45}
(\widehat{\mathcal{H}}_\e-\l_{\e,N}^{(n,m)}\mathrm{p}_\e)
\psi_{\e,N}^{(n,m)}=h_{\e,N}^{(n,m)},
\end{equation}
where the right-hand side obeys an inequality
\begin{equation}\label{3.46}
\|h_{\e,N}^{(n,m)}\|_{C^k(\overline{\Om})}\leqslant C_{N,k}^{(n,m)}\e^{N-1},
\end{equation}
with constants $C_{N,k}^{(n,m)}>0$ independent of $\e$.
\end{lmm}

We proceed to the justification of the formal asymptotics (\ref{3.1}),
(\ref{3.2}). We use the standard approach based on Lemmas~12,~13 from \cite{VL}.
More precisely, we use these lemmas in the formulation presented in Lemma~1.1 in
\cite[Ch. I\!I\!I, Sec. 1.1]{IOS}. For the convenience of the reader below we
give the mentioned lemma.

\begin{lmm}\label{lm3.3}
Let $\mathcal{A}: H\to H$ be a continuous linear compact
self-adjoint operator in a Hilbert space $H$. Suppose that there
exist a real $\mu>0$ and a vector $u\in H$, such that $\|u\|_{H}=1$
and
\begin{equation*}
\|\mathcal{A}u-\mu u\|_{H}\leqslant \a,\quad \a=const>0.
\end{equation*}
Then there exists an eigenvalue $\mu_i$ of operator $\mathcal{A}$
such that
\begin{equation*}
|\mu_i-\mu|\leqslant \a.
\end{equation*}
Moreover, for any $d>a$ there exists a vector $\overline{u}$ such
that
\begin{equation*}
\|u-\overline{u}\|_{H}\leqslant 2\a d^{-1},\quad \|\overline{u}\|_{H}=1,
\end{equation*}
and $\overline{u}$ is a linear combination of eigenvectors of
operator $\mathcal{A}$ corresponding to eigenvalues of $\mathcal{A}$
from the segment $[\mu-d,\mu+d]$.
\end{lmm}

Let $\e$ be small enough. Denote
$\widehat{\psi}_{\e,N}^{(n,m)}:=\mathrm{p}_\e^{1/2}\psi_{\e,N}^{(n,m)}$. We
rewrite (\ref{3.45}) as
\begin{equation}\label{3.35}
\mathrm{p}_\e^{-1/2} \widehat{\mathcal{H}}_\e\mathrm{p}_\e^{-1/2}
\widehat{\psi}_{\e,N}^{(n,m)}=\l_{\e,N}^{(n,m)}\widehat{\psi}_{\e,N}^{(n,m)}
+\mathrm{p}_\e^{-1/2}h_{\e,N}^{(n,m)}.
\end{equation}
The identities (\ref{2.2a}), (\ref{2.2}), (\ref{2.3}) imply that the
operator $\widehat{\mathcal{H}}_\e$ is self-adjoint and
non-negative. The same is obviously true for $\mathrm{p}_\e^{-1/2}
\widehat{\mathcal{H}}_\e \mathrm{p}_\e^{-1/2}$. Let $\d>0$ be a
positive number. Then the operator $\mathcal{A}_\e:=(1+\d
\mathrm{p}_\e^{-1/2} \widehat{\mathcal{H}}_\e
\mathrm{p}_\e^{-1/2})^{-1}$ is well-defined as an operator in
$L_2(\Om)$, is bounded and self-adjoint, and satisfies the estimate
\begin{equation}\label{3.36}
\|\mathcal{A}_\e\|\leqslant 1.
\end{equation}
As it follows from (\ref{2.2}), for any $f\in L_2(\Om)$ the function
$v=\mathcal{A}_\e f$ is a generalized solution to the boundary value
problem
\begin{align*}
-\d&\mathrm{p}_\e^{-1/2} \sum\limits_{i,j=1}^{3}
\frac{\p}{\p\xi_i}\widetilde{A}_{ij}^{(\e)}
\frac{\p}{\p\xi_j}\mathrm{p}_\e^{-1/2} v +v=f\quad\text{in}\quad \Om,\qquad
v=0\quad\text{on}\quad\p\Om.
\\
&\xi_1:=s,\quad  \widetilde{A}_{11}^{(\e)}:=A_{11}^{(\e)},\quad
\widetilde{A}_{i1}^{(\e)}:=\widetilde{A}_{1i}^{(\e)}:=\e^{-1}
A_{i1}^{(\e)},\quad \widetilde{A}_{ij}^{(\e)}:=\e^{-2} A_{ij}^{(\e)},\quad
i,j=2,3.
\end{align*}
Hence, the operator $\mathcal{A}_\e$ is also bounded as that from
$L_2(\Om)$ into $\H^1(\Om)$. In view of the compact embedding of
$\H^1(\Om)$ in $L_2(\Om)$ the operator $\mathcal{A}_\e$ is compact
as that in $L_2(\Om)$.

We rewrite (\ref{3.35}) as
\begin{equation*}
(1+\d\l_{\e,N}^{(n,m)})^{-1}\widehat{\psi}_{\e,N}^{(n,m)}
=\mathcal{A}_\e \widehat{\psi}_{\e,N}^{(n,m)}+\d
\widehat{h}_{\e,N}^{(n,m)},\quad
\widehat{h}_{\e,N}^{(n,m)}:=(1+\d\l_{\e,N}^{(n,m)})^{-1}
\mathcal{A}_\e \mathrm{p}_\e^{-1/2} h_{\e,N}^{(n,m)}.
\end{equation*}
It follows from (\ref{3.24a}), (\ref{3.46}), (\ref{3.36}) that we
can choose $\d=\d(\e,N,n,m)$ so that
\begin{equation}\label{3.38}
1\leqslant 1+\d\l_{\e,n}^{(n,m)}\leqslant 2, \quad
\frac{\|\widehat{h}_{\e,N}^{(n,m)}\|_{L_2(\Om)}}{
\|\widehat{\psi}_{\e,N}^{(n,m)}\|_{L_2(\Om)}} \leqslant
\widehat{C}_N^{(n,m)}\d\e^{N-1}\leqslant \frac{1}{2},
\end{equation}
where $\widehat{C}_N^{(n,m)}$ are some constants independent of
$\e$.

We apply Lemma~\ref{lm3.3} with
\begin{equation}\label{3.38a}
\begin{aligned}
&H=L_2(\Om),\quad \mathcal{A}=\mathcal{A}_\e,\quad
\mu=(1+\d\l_{\e,N}^{(n,m)})^{-1},
\\
&u=\frac{\widehat{\psi}_{\e,N}^{(n,m)}}
{\|\widehat{\psi}_{\e,N}^{(n,m)}\|_{L_2(\Om)}}, \quad \a=
\widehat{C}_N^{(n,m)}\e^{N-1}\d,
\end{aligned}
\end{equation}
and conclude that there exists an eigenvalue $\mu_{\e,N}^{(n,m)}$ of
$\mathcal{A}_\e$ such that
\begin{equation}\label{3.39}
|\mu_{\e,N}^{(n,m)}-(1+\d\l_{\e,N}^{(n,m)})^{-1}|\leqslant
\widehat{C}_N^{(n,m)}\d\e^{N-1},
\end{equation}
where $C_N^{(n,m)}$ are some positive constants independent of $\e$.
It is clear that
$\widetilde{\l}_{\e,N}^{(n,m)}:=\big((\mu_{\e,N}^{(n,m)})^{-1}-1\big)\d^{-1}$
is an eigenvalue of $\mathrm{p}_\e^{-1/2}\widehat{\mathcal{H}}_\e
\mathrm{p}_\e^{-1/2}$. It follows from (\ref{3.38}), (\ref{3.39})
that
\begin{equation*}
\frac{1}{1+\d\widetilde{\l}_{\e,N}^{(n,m)}}\geqslant \frac{1}{1+\d
\l_{\e,N}^{(n,m)}}-\widehat{C}_N^{(n,m)}\e^{N-1}\d\geqslant
\frac{1}{2},\quad 1+\d\widetilde{\l}_{\e,N}^{(n,m)}\geqslant 2.
\end{equation*}

The last inequality and (\ref{3.38}), (\ref{3.39}) imply
\begin{gather}
\big|(1+\d\widetilde{\l}_{\e,N}^{(n,m)})-
(1+\d\l_{\e,N}^{(n,m)})\big|\leqslant
\widehat{C}_N^{(n,m)}\d\e^{N-1}
|1+\d\widetilde{\l}_{\e,N}^{(n,m)}|\, |1+\d\l_{\e,N}^{(n,m)}|,
\nonumber
\\
|\widetilde{\l}_{\e,N}^{(n,m)})- \l_{\e,N}^{(n,m)}|\leqslant
4\widehat{C}_N^{(n,m)}\d\e^{N-1}. \label{3.40}
\end{gather}
By $\e^{(n,m)}_N$ we denote a monotonically decreasing (in $N$)
sequence such that
\begin{equation*}
\widehat{C}_N^{(n,m)}\e\leqslant \widehat{C}_{N-1}^{(n,m)}\quad
\text{as}\quad \e\leqslant \e_N^{(n,m)}.
\end{equation*}
Letting
\begin{equation*}
\l_\e^{(n,m)}:=\widetilde{\l}_{\e,N}^{(n,m)}\quad \text{for}\quad
\e\in[\e_N^{(n,m)},\e_{N+1}^{(n,m)}),
\end{equation*}
and employing (\ref{3.40}), we see that the eigenvalue
$\l_\e^{(n,m)}$ of $\mathrm{p}_\e^{-1/2}\widehat{\mathcal{H}}_\e
\mathrm{p}_\e^{-1/2}$ has the asymptotic expansion (\ref{1.1}). To
complete the proof it remains to note that the eigenvalues of
$\mathrm{p}_\e^{-1/2}\widehat{\mathcal{H}}_\e \mathrm{p}_\e^{-1/2}$
coincide with those of $\mathcal{H}_\e$.

\section{Proof of Theorem~\ref{th1.2}}

We begin the proof with the result of Theorem~4.4 in \cite{BMT}.
Namely, the item~(ii) of this theorem says that given any $M>0$,
there exists $\e_0(M)>0$ so that for all $\e<\e_0$ the first $M$
eigenvalues $\l_m(\e)$, $m=1,\ldots,M$, of $\mathcal{H}_\e$ taken
counting multiplicities satisfy the asymptotics
\begin{equation}\label{4.0a}
\l_m(\e)=\e^{-2}\l_1+\l_0^{(1,m)}+o(1),\quad m=1,\ldots,M.
\end{equation}
Since the eigenvalues $\l_0^{(1,m)}$ of $\mathcal{L}_1$ are simple,
the same is true for the eigenvalues $\l_m(\e)$.

It follows from (\ref{4.0a}) that there exists a fixed number
$\tht>0$ so that for each $m=1,\ldots,M$, and all $\e<\e_0(M)$ the
interval
\begin{equation}\label{4.0b}
(\e^{-2}\l_1+\l_0^{(1,m)}-\tht,\e^{-2}\l_1+\l_0^{(1,m)}+\tht)
\end{equation}
contains exactly one eigenvalue of $\mathcal{H}_\e$ which is
$\l_m(\e)$. In accordance with Theorem~\ref{th1.1} the eigenvalues
$\l^{(1,m)}_\e$ of $\mathcal{H}_\e$ satisfy the same asymptotics as
$\l_m(\e)$. Hence, for sufficiently small $\e_0(M)$ and  all
$m=1,\ldots,M$ each of the intervals (\ref{4.0b}) contains the
eigenvalue $\l^{(1,m)}(\e)$. Therefore, $\l^{(1,m)}(\e)=\l_m(\e)$,
and it proves the statement of the theorem on the eigenvalues.

To prove the statement on the eigenfunctions, we adopt the same
notations as in the proof of Theorem~\ref{th1.1}. We again apply
Lemma~\ref{lm3.3} with (\ref{3.38a}) and we take $d=\e^{N/2}$. Then
there exists a linear combination $\widetilde{\psi}_{\e,N}^{(1,m)}$
of the eigenfunctions associated with the eigenvalues of
$\mathcal{A}_\e$ lying in the segment
\begin{equation}\label{4.2}
[(1+\d\l_{\e,N}^{(1,m)})^{-1}-\e^{N/2},(1+\d\l_{\e,N}^{(1,m)})^{-1}+\e^{N/2}]
\end{equation}
such that
\begin{equation}\label{4.1}
\big\|\widetilde{\psi}_{\e,N}^{(1,m)}\|
\mathrm{p}_\e^{1/2}\psi_{\e,N}^{(1,m)} \|_{L_2(\Om)} -
\mathrm{p}_\e^{1/2}\psi_{\e,N}^{(1,m)}\big\|_{L_2(\Om)}\leqslant
C_N^{(m)}\e^{N/2-1}\d,
\end{equation}
where $C_N^{(m)}$ are positive constants independent of $\e$ and
$\d$. We choose $\d=\e^2$ and by the asymptotics (\ref{1.1}) for
$\l_\e^{(1,m)}$ we obtain the estimate
\begin{equation*}
\left| \frac{1}{1+\d\l_\e^{(1,m)}}- \frac{1}{1+\d\l_\e^{(1,p)}}
\right|=\frac{\d|\l_\e^{(1,m)}-\l_\e^{(1,p)}|}{|1+\d\l_\e^{(1,m)}|\,
|1+\d\l_\e^{(1,p)}|}\geqslant C\e^2
\end{equation*}
for $\e$ small enough, $m,p=1,\ldots,M$, $m\not=p$,  where $C$ is a
positive constant independent of $\e$, $m$, and $p$. Hence, for
$N\geqslant 5$, $m=1,\ldots,M$, and $\e$ small enough the intervals
(\ref{4.2}) contain no eigenvalues of $\mathcal{A}_\e$ except
$(1+\e^2\l_\e^{(1,m)})^{-1}$. This eigenvalue is simple, since the
corresponding eigenvalue $\l_\e^{(1,m)}$ is simple. Thus, the linear
combination $\widetilde{\psi}_{\e,N}^{(1,m)}$ is an orthonormalized
in $L_2(\Om)$ eigenfunction associated with
$(1+\e^2\l_\e^{(1,m)})^{-1}$. Moreover, it is independent of $N$.

By the defintion of $\widehat{\psi}_\e^{(1,m)}$ and
Lemma~\ref{lm3.1} we have
\begin{equation*}
\| \mathrm{p}_\e^{1/2}\psi_{\e,N}^{(1,m)} \|_{L_2(\Om)}=
\sum\limits_{j=0}^{N} c_j^{(m)}\e^j+\Odr(\e^{N+1}),\quad
c_j^{(m)}=const,
\end{equation*}
for all $N\geqslant 0$, $m=1,\ldots,M$. Hence, there exists a
function $c_m=c_m(\e)$ such that
\begin{equation*}
c_m(\e)=\| \mathrm{p}_\e^{1/2}\psi_{\e,N}^{(1,m)}
\|_{L_2(\Om)}+\Odr(\e^{N+1})\quad\text{for all}\quad N\geqslant 0.
\end{equation*}
The identity obtained and (\ref{4.1}) yield
\begin{equation*}
\|\widetilde{\psi}_{\e,*}^{(1,m)}-\mathrm{p}_\e^{1/2}
\psi_{\e,N}^{(1,m)}\|_{L_2(\Om)} =\Odr(\e^{N/2+1}),\quad
\widetilde{\psi}_{\e,*}^{(1,m)}:=c_m(\e)\widetilde{\psi}_{\e,N}^{(1,m)}.
\end{equation*}
Denoting $\psi_\e^{(1,m)}(s,\xi):=\mathrm{p}_\e^{-1/2}(s,\xi)
\widetilde{\psi}_{\e,*}^{(1,m)}(s,\xi)$, we can rewrite the last
equation as
\begin{equation}\label{4.9}
\|\Phi_{\e,N}^{(m)}\|_{L_2(\Om)}=\Odr(\e^{N/2+1}),\quad
\Phi_{\e,N}^{(m)}:=\psi_\e^{(1,m)}-\psi_{\e,N}^{(1,m)}.
\end{equation}
It is also clear that the function $\psi_\e^{(1,m)}$ solves
(\ref{2.4}) with $\l_\e=\l_\e^{(1,m)}$. Hence, by (\ref{3.45})
\begin{equation}\label{4.9a}
\widehat{\mathcal{H}}_\e\Phi_{\e,N}^{(m)}=
\l_\e^{(1,m)}\mathrm{p}_\e\Phi_{\e,N}^{(m)}
+\widetilde{h}_{\e,N}^{(m)},\quad
\widetilde{h}_{\e,N}^{(m)}:=(\l_\e^{(1,m)}-
\l_{\e,N}^{(1,m)})\mathrm{p}_\e\psi_{\e,N}^{(1,m)}-h_{\e,N}^{(1,m)}.
\end{equation}
Due to this equation we can write the integral identity
\begin{equation*}
(\widehat{\mathcal{H}}_\e\Phi_{\e,N}^{(m)},
\Phi_{\e,N}^{(m)})_{L_2(\Om)} =
\l_\e^{(1,m)}(\mathrm{p}_\e\Phi_{\e,N}^{(m)},\Phi_{\e,N}^{(m)})_{L_2(\Om)}
+(\widetilde{h}_{\e,N}^{(m)}, \Phi_{\e,N}^{(m)})_{L_2(\Om)}.
\end{equation*}
From (\ref{4.9}), (\ref{3.46}), and (\ref{1.1}) we derive
\begin{equation}\label{4.11}
\|\widetilde{h}_{\e,N}^{(m)}\|_{C^k(\overline{\Om})}=\Odr(\e^{N+1}),\quad
N\geqslant 5, \quad k\geqslant 0.
\end{equation}
Together with (\ref{2.2}), (\ref{2.3}), (\ref{4.9}) it gives
\begin{equation}\label{4.12}
\|\nabla_{(s,\xi)}\Phi_{\e,N}^{(m)}\|_{L_2(\Om)}^2\leqslant C
(\mathrm{A}^{(\e)}\nabla_{(s,\xi)}\Phi_{\e,N}^{(m)},
\nabla_{(s,\xi)}\Phi_{\e,N}^{(m)})_{L_2(\Om)}=\Odr(\e^{N-4}), \quad
m=1,\ldots,M.
\end{equation}
Combining (\ref{4.9}) and (\ref{4.12}), we conclude that the
asymptotics (\ref{1.3}) hold true in $\H^1(\Om)$-norm.

We proceed to the proof of (\ref{1.3}) in $C^k(\Om^{(t)})$-norm.
First we note that by the standard smoothness improving theorems we
have $\psi_\e^{(1,m)}\in C^\infty(\overline{\Om^{(t)}})$ for all
$t\in(0,s_0/2)$. The rest of the proof follows from (\ref{4.9}),
(\ref{4.9a}), (\ref{4.11}), (\ref{4.12}), the embedding of
$\H^{k+2}(\Om)$ into $C^k(\overline{\Om})$, and from the next lemma
applied to $\Phi_{\e,N}^{(m)}$.

\begin{lmm}\label{lm4.1}
Let $u_\e$ be a solution to the equation
\begin{equation*}
\widehat{\mathcal{H}}_\e u_\e=\l_\e^{(1,m)} \mathrm{p}_\e
u_\e+h,\quad h\in C^\infty(\overline{\Om}).
\end{equation*}
Then $u_\e\in C^\infty(\overline{\Om^{(t)}})$ for all $t\in
(0,s_0/2)$, and
\begin{equation}\label{4.18}
\|u_\e\|_{\H^k(\overline{\Om^{(t)}})} \leqslant C_k\e^{-2(k-1)}
\Big(\|u_\e\|_{\H^1(\Om)}+\|h\|_{\H^{k-2}(\Om)}\Big)
\end{equation}
for all $k\geqslant 2$, $t\in(0,s_0/2)$, where the constants $C_k$
are independent of $\e$, $h$ and $u_\e$.
\end{lmm}

\begin{proof}
The smoothness of $u_\e$ follows from that of $A_{ij}^{(\e)}$, $h$,
$\mathrm{p}_\e$, and the smoothness improving theorems. For the sake
of brevity we denote $\xi_1:=s$, $\xi:=(\xi_1,\xi_2,\xi_3)$.

Let $\chi_1^{(t)}=\chi_1^{(t)}(\xi_1)$ be an infinitely
differentiable cut-off function equalling one as $\xi_1\in[t,s_0-t]$
and vanishing for $\xi_1\in[0,t/2]\cup[s_0-t/2,s_0]$. We fix
$t\in(0,s_0/2)$ and denote
\begin{equation*}
u_\e^{(t)}(\xi):=\chi_1^{(t)}(\xi_1)u_\e(\xi).
\end{equation*}
It is straightforward to check that $u_\e^{(t)}$ solves the equation
\begin{equation}\label{4.20}
\widehat{\mathcal{H}}_\e u_\e^{(t)}=\l_\e^{(1,m)}\mathrm{p}_\e
u_\e^{(t)} +h\chi_1^{(t)}+\e^{-2}G_1^{(t)}\left(\e,\xi,u_\e,\frac{\p
u_\e }{\p\xi_i}\right).
\end{equation}
The symbol $\frac{\p u_\e }{\p\xi_i}$ in the arguments of
$G_1^{(t)}$ indicates that this function depends on all first
derivatives of $u_\e$. The function $G_1^{(t)}$ is linear with
respect to $u_\e$ and the first derivatives of $u_\e$. The
coefficients at $u_\e$ and at the first derivatives of $u_\e$ belong
to $C^\infty(\overline{\Om})$ and are bounded in
$C^\k(\overline{\Om})$-norms uniformly in $\e$ for all $k\geqslant
0$. We differentiate (\ref{4.20}) with respect to $\xi_1$,
\begin{equation}\label{4.21}
\begin{aligned}
\widehat{\mathcal{H}}_\e \frac{\p u_\e^{(t)}}{\p\xi_1}
=&\l_\e^{(1,m)}\mathrm{p}_\e \frac{\p u_\e^{(t)}}{\p \xi_1}+
\e^{-2}G_2^{(t)}\left(\e,\xi,u_\e^{(t)},\frac{\p u_\e^{(t)}
}{\p\xi_i},\frac{\p^2 u_\e^{(t)}
}{\p\xi_i\p\xi_j}\right)
\\
&+\frac{\p}{\p \xi_1}
\left(h\chi_1^{(t)}+\e^{-2}G_1^{(t)}\left(\e,\xi,u_\e,\frac{\p u_\e
}{\p\xi_i}\right)\right).
\end{aligned}
\end{equation}
The symbols $\frac{\p u_\e^{(t)} }{\p\xi_i}$, $\frac{\p^2 u_\e^{(t)}
}{\p\xi_i\p\xi_j}$ in the arguments of $G_1^{(t)}$ indicates that
this function depends on all first and second derivatives of
$u_\e^{(t)}$. The function $G_1^{(t)}$ is linear with respect to
$u_\e$ and the first derivatives of $u_\e$. The coefficients at
$u_\e$ and at the first derivatives of $u_\e$ belong to
$C^\infty(\overline{\Om})$ and are bounded in
$C^\k(\overline{\Om})$-norms uniformly in $\e$ for all $k\geqslant
0$.

Starting from (\ref{4.21})  and proceeding as in (\ref{4.12}), we
obtain
\begin{equation}\label{4.22}
\begin{aligned}
\Big\|\nabla  &\frac{\p u_\e^{(t)}}{\p\xi_1}
\Big\|_{L_2(\Om)}^2\leqslant  C\left(\mathrm{A}^{(\e)}\nabla
\frac{\p u_\e^{(t)}}{\p\xi_1}, \nabla \frac{\p u_\e^{(t)}}{\p\xi_1
}\right)_{L_2(\Om)}
\\
=&C\l_\e^{(1,m)}\left(\mathrm{p}_\e\frac{\p u_\e^{(t)}}{\p\xi_1},
\frac{\p u_\e^{(t)}}{\p\xi_1}\right)_{L_2(\Om)}
\\
&-C \left( h\chi_1^{(t)}+\e^{-2}G_1^{(t)}\left(\e,\xi,u_\e,\frac{\p
u_\e }{\p\xi_i}\right),\frac{\p^2 u_\e^{(t)}}{\p\xi_1^2}
\right)_{L_2(\Om)}
\\
&+\e^{-2}\left( G_2^{(t)}\left(\e,\xi,u_\e^{(t)},\frac{\p u_\e^{(t)}
}{\p\xi_i},\frac{\p^2 u_\e^{(t)} }{\p\xi_i\p\xi_j}\right),\frac{\p
u_\e^{(t)}}{\p\xi_1} \right)_{L_2(\Om)}
\\
\leqslant&
C\e^{-4}\left(\|h\|_{L_2(\Om)}^2+\|u_\e\|_{\H^1(\Om)}^2\right)+\frac{1}{2}
\Big\|\nabla  \frac{\p u_\e^{(t)}}{\p\xi_1} \Big\|_{L_2(\Om)}^2.
\end{aligned}
\end{equation}
Here and till the end of the proof by $C$ we indicate non-specific
constants independent of $\e$, $u_\e$ and $h$. The obtained estimate
implies
\begin{equation}\label{4.22a}
\Big\|\nabla \frac{\p u_\e^{(t)}}{\p\xi_1} \Big\|_{L_2(\Om)}
\leqslant C\e^{-2}\left(\|h\|_{L_2(\Om)} +\|u_\e\|_{\H^1(\Om)}
\right).
\end{equation}

Let us estimate $\Big\|\nabla \frac{\p u_\e^{(t)}}{\p\xi_i}
\Big\|_{L_2(\Om)}$, $i=2,3$. We could have tried to differentiate
(\ref{4.20}) with respect to $\xi_i$ and proceed as above. However,
the function $\frac{\p u_\e^{(t)}}{\p\xi_i}$ does not vanish on
$(0,s_0)\times\p\om$ and this is the main difficulty. This is why we
have to employ a slightly different trick. We introduce an
infinitely differentiable in $\overline{\om}$ cut-off function
$\chi_2=\chi_2(\xi)$ equalling one in a small neighborhood of
$\p\om$ and vanishing outside a bigger neighborhood. Writing the
equation for $u_\e^{(t)}(1-\chi_2)$ similar to (\ref{4.20}) and
proceeding as in (\ref{4.21}), (\ref{4.22}), one can show that
\begin{equation}\label{4.23}
\left\|\frac{\p}{\p\xi_i} u_\e^{(t)}(1-\chi_2)\right\|_{\H^1(\Om)}
\leqslant C\e^{-2}\big(\|h\|_{L_2(\Om)}+\|u_\e\|_{\H^1(\Om)} \big).
\end{equation}

In a small neighborhood of $\p\om$ we introduce new variables
$\z=(\z_1,\z_2,\z_3)$, where $\z_1=\xi_1$, $\z_2$ is the arc length
of $\p\om$ and $\z_3$ is the distance from the point to $\p\om$
measured in the direction of the inward normal. Then it follows from
(\ref{4.20}) that the function $v_\e^{(t)}=\chi_2 u_\e^{(t)}$
satisfies the boundary value problem
\begin{equation}\label{4.24}
\begin{aligned}
-\sum\limits_{i,j=1}^{3} B_{ij}^{(\e)} \frac{\p^2
v_\e^{(t)}}{\p\z_i\p \z_j}=&h\chi_2\chi_1^{(t)} +\e^{-2}\chi_2
G_1^{(t)}+B_0^{(\e)}u_\e^{(t)} +\sum\limits_{i=1}^{3} B_i^{(\e)}
\frac{\p u_\e^{(t)}}{\p\z_i},
\\
&\z\in(0,s_0)\times(0,z_2^{(0)})\times(0,\z_3^{(0)}),
\end{aligned}
\end{equation}
where $\z_2^{(0)}$ is the length of $\p\om$, $\z_3^{(0)}$ is a small
fixed number. The operator in the left hand side of the last
equation is elliptic uniformly in $\z$ and $\e$. The function
$v_\e^{(t)}$ satisfies periodic boundary condition as $\z_2=0$ and
$\z_2=\z_2^{(0)}$, and vanishes as $\z_3=0$, $\z_3=\z_3^{(0)}$,
$\z_1=0$, $\z_1=s_0$. The coefficients $B_{ij}^{(\e)}$, $B_i^{(\e)}$
are infinitely differentiable and satisfy the estimates
\begin{equation*}
\|B_{ij}^{(\e)}\|_{C^1(
[0,s_0]\times[0,\z_2^{(0)}]\times[0,\z_3^{(0)}])} \leqslant
C\e^{-2},\quad \|B_i^{(\e)}\|_{C^1(
[0,s_0]\times[0,\z_2^{(0)}]\times[0,\z_3^{(0)}])} \leqslant
C\e^{-2}.
\end{equation*}

We differentiate (\ref{4.24}) with respect to $\z_2$, and in the
same fashion as in (\ref{4.21}), (\ref{4.22}), (\ref{4.22a}) we
obtain
\begin{equation}\label{4.25}
\Big\|\nabla \frac{\p v_\e^{(t)}}{\p\z_2} \Big\|_{L_2(
[0,s_0]\times[0,\z_2^{(0)}]\times[0,\z_3^{(0)}] )} \leqslant
C\e^{-2}\left(\|h\|_{L_2(\Om)} +\|u_\e\|_{\H^1(\Om)} \right).
\end{equation}
Now we express the term $\frac{\p^2 v_\e^{(t)}}{\p\z_3^2}$ from
(\ref{4.24}). Together with (\ref{4.22a}), (\ref{4.23}),
(\ref{4.25}) it gives the estimate
\begin{equation*}
\Big\| \frac{\p^2v_\e^{(t)}}{\p\z_3^2} \Big\|_{L_2(
[0,s_0]\times[0,\z_2^{(0)}]\times[0,\z_3^{(0)}] )} \leqslant
C\e^{-2}\left(\|h\|_{L_2(\Om)} +\|u_\e\|_{\H^1(\Om)} \right).
\end{equation*}
This estimate and (\ref{4.22a}), (\ref{4.23}), (\ref{4.25}) imply
(\ref{4.18}) for $k=1$. To prove it for other $k$'s, it is
sufficient to proceed as above starting with differentiating the
equations for $\frac{\p u_\e^{(t)}}{\p\xi_1}$,
$\frac{\p}{\p\z_i}(1-\chi_2)u_\e^{(t)}$, $\frac{\p}{\p\z_i}\chi_2
u_\e^{(t)}$, $i=2,3$.
\end{proof}

\section*{Acknowledgments}

We thank D. Krej\v ci\v r\'\i k and the referees for valuable
remarks which allowed us to improve the original version of the
paper.

\end{document}